\RequirePackage[dvipsnames]{xcolor}
\RequirePackage{tcolorbox}
\documentclass{imsart}
\usepackage{amsthm,amsmath}
\usepackage[numbers]{natbib}
\usepackage{fontenc}
\usepackage{amsfonts}
\usepackage{mathtools}
\usepackage{enumitem}
\usepackage{tikz}

\usepackage{tcolorbox}

\newtcolorbox{mybox}[1]{colback=red!5!white, colframe=red!75!black,fonttitle=\bfseries, title=#1}

\def\ML{\MoveEqLeft}
\definecolor{mycolor}{rgb}{0.15,0.56,0.25}

\usepackage[capitalize,nameinlink]{cleveref}
\usepackage[final,notref,notcite]{showkeys}

\newcommand{\ignore}[1]{}{}
\newcommand{\hatK}{\hat{K}}
\newcommand{\lbl}{\label}

\newcommand{\nn}{}

\newtheorem{theorem}{Theorem}[section]

\newtheorem{corollary}{Corollary}[section]
\newtheorem{proposition}{Proposition}[section]
\newtheorem{lemma}{Lemma} [section]
\theoremstyle{remark}
\newtheorem{remark}{Remark}[section]
\newtheorem{example}{Example}[section]

\newlist{condition}{enumerate}{10}
\setlist[condition]{label*=\textrm{({A}\arabic*)},ref=\textup{({A}\arabic*)}}
\newlist{condition2}{enumerate}{10}
\setlist[condition2]{label*=(\arabic{section}.\arabic{subsection}.\roman*)}

\newlist{condition3}{enumerate}{10}
\setlist[condition3]{label*=({B}\arabic*)}

\crefdefaultlabelformat{#2\textup{#1}#3}

\crefname{equation}{}{}
\crefformat{equation}{\textup{(#2#1#3)}}
\crefrangeformat{equation}{\textup{(#3#1#4)--(#5#2#6)}}

\crefname{subsection}{Subsection}{Subsections}

\crefname{conditioni}{condition}{conditions}
\Crefname{conditioni}{Condition}{Conditions}

\crefname{condition2i}{condition}{conditions}
\Crefname{condition2i}{Condition}{Conditions}

\crefname{condition3i}{condition}{conditions}
\Crefname{condition3i}{Condition}{Conditions}
\newcommand{\eq}[1]{ (\ref{#1})}
\newcommand{\beqn}{\begin{eqnarray}}
\newcommand{\eeqn}{\end{eqnarray}}
\newcommand{\beq}{\begin{eqnarray*}}
\newcommand{\eeq}{\end{eqnarray*}}
\newcommand{\bea}{\begin{eqnarray}}
\newcommand{\eea}{\end{eqnarray}}
\newcommand{\bei}{\begin{itemize}}
\newcommand{\eei}{\end{itemize}}
\newcommand{\ben}{\begin{enumerate}}
\newcommand{\een}{\end{enumerate}}
\newcommand{\bet}{\begin{theorem}}
\newcommand{\eet}{\end{theorem}}
\newcommand{\bel}{\begin{lemma}}
\newcommand{\eel}{\end{lemma}}
\newcommand{\bep}{\begin{proposition}}
\newcommand{\eep}{\end{proposition}}
\newcommand{\bed}{\begin{definition}}
\newcommand{\eed}{\end{definition}}
\newcommand{\bec}{\begin{corollary}}
\newcommand{\eec}{\end{corollary}}
\newcommand{\bex}{\begin{example}}
\newcommand{\eex}{\end{example}}

\newcommand{\sgn}{\mathop{\rm sgn}\nolimits}

\def\lme{\left(}
\def\rme{\right)}

\def\w0{ \I(W>0) }
\def\intud{\int_{|u|\leq \delta}}
\def\inttd{\int_{|t|\leq \delta}}
\def\0{\boldsymbol{0}}
\def\F{\boldsymbol{F}}

\def\R{\mathbb{R}}

\def\D{\boldsymbol{\textbf{D}}}
\def\A{\boldsymbol{A}}

\def\I{I}

\def\0{\boldsymbol{0}}

\def\E{\mathop{}\!\mathrm{E}\,\!}
\def\P{{\rm P}\,\!}

\def\M{\mathcal{M}}

\newcommand{\conep}[3][]{\E_{#1}\left( #2 \,\middle\vert\, #3  \right)}
\newcommand{\conepb}[3][]{\E_{#1}\bigl( #2 \bigm\vert #3  \bigr)}

\newcommand{\expb}[1]{ {\exp} ( #1 ) }
\def\dd{\mathop{}\!\mathrm{d}}

\newenvironment{equ}
{\begin{equation} \begin{aligned}}
{\end{aligned} \end{equation}}
\AfterEndEnvironment{equ}{\noindent\ignorespaces}
\def\cite{\citet*}
\begin{document}


\begin{frontmatter}
\title{Cram\'er-type Moderate Deviation Theorems for Nonnormal Approximation} 
\runtitle{Moderate deviation for nonnormal approximation}









\begin{aug}
\author{\fnms{Qi-Man}
\snm{Shao$^{1,2}$}\protect\ead[label=e1]{qmshao@cuhk.edu.hk}\thanksref{t1}},
\thankstext{t1}{Research partially supported by Hong Kong RGC GRF 14302515 and 14304917}
\author{\fnms{Mengchen}
\snm{Zhang$^{3}$}\protect\ead[label=e2]{mzhangag@connect.ust.hk}}
and

\author{\fnms{Zhuo-Song}
\snm{Zhang$^{2,4}$}\protect\ead[label=e3]{zszhang.stat@gmail.com}\thanksref{t2}}
\thankstext{t2}{Corresponding author. Research partially supported by Singapore Ministry of Education Academic Research Fund MOE 2018-T2-076}


\runauthor{Q.M. Shao, M.C. Zhang and Z.S. Zhang}

\affiliation{Southern University of Science and Technology $^{1}$,  The Chinese University of Hong Kong $^{2}$,
Hong Kong UST $^{3}$ and National University of Singapore $^4$}

\address{Department of Statistics and Data Scinece\\
Southern University of Science and Technology\\
Shenzhen, Guangdong, P.R. China\\
and \\
Department of Statistics\\
The Chinese University of Hong Kong \\
Shatin, N.T.,
Hong Kong\\
\printead{e1}\\
\phantom{E-mail:\ }
}

\address{Department of Mathematics\\
Hong Kong University of Science and Technology\\
Clear Water Bay\\
Kowloon, Hong Kong\\
\printead{e2}\\
\phantom{E-mail:\ }}

\address{Department of Statistics\\
The Chinese University of Hong Kong\\
Shatin, N.T., Hong Kong\\
and \\
Department of Statistics and Applied Probability\\
National University of Singapore\\
Singapore 117546 \\
\printead{e3}\\
\phantom{E-mail:\ }
}

\end{aug}

\begin{abstract}
    A Cram\'er-type moderate deviation theorem quantifies the relative error of the tail probability approximation. It provides a criterion whether the limiting tail probability can be used to estimate the tail probability under study. Chen, Fang and Shao (2013) obtained a general Cram\'er-type moderate result using Stein's method when the limiting was a normal distribution. In this paper,
Cram\'er-type moderate deviation theorems are established for nonnormal approximation under a general Stein identity, which is satisfied via the exchangeable pair approach and Stein's coupling. In particular, a Cram\'er-type moderate deviation theorem is obtained for the general Curie--Weiss model and the imitative monomer-dimer mean-field model.
\end{abstract}



\end{frontmatter}

\section{Introduction}

Consider a sequence of random variables $W_n$. One often needs
 to calculate the tail probability of $W_n$ such as
$\P(W_n \geq x_n)$. Since the exact distribution of $W_n$ is hardly known, it is common to use the
limiting distribution, that is, assuming that $W_n$ converges to $Y$ in distribution, $\P(Y\geq x_n)$ is used to estimate \(\P (W_n \geq x_n)\). The Cram\'er-type moderate deviation seeks the largest possible \(a_n\) so that
\begin{equation}
    {\frac{\P(W_n \geq x)}{\P(Y \geq x)} }
= 1 + \mbox{error} \to 1 \label{01}
\end{equation}
holds for $ 0 \leq x \leq a_n$. 
This quantifies the relative error of the distribution approximation and provides a
criterion whether the limiting tail probability can be used to estimate the tail probability.
When $Y$ is the normal random variable and $W_n$ is the standardized sum of the independent random variables, the Cram\'er-type moderate deviation is well understood. In particular,
for independent and identically distributed random variables $X_1, \ldots, X_n$ with $\E X_i =0, \E X_i^2 =1$ and
$\E e^{t_0 \sqrt{|X_1|}} < \infty$, \(t_0 >0\), it holds that
\begin{equation}
\frac{\P(W_n \geq x) }{ 1- \Phi(x)} = 1 + O(1) ( 1+x^3) /\sqrt{n}
\lbl{02}
\end{equation}
for $ 0 \leq x \leq n^{1/6}$, where $W_n = (X_1 + \cdots + X_n)/\sqrt{n}$.  The finite-moment-generating function of $|X_1|^{1/2}$ is necessary, and both the range $0 \leq x \leq n^{1/6}$ and the order of the error term $(1+x^3)/\sqrt{n}$ are optimal. We refer to \cite{Li61P} and \cite[p. 251]{Pe75S} for details.

Considering general dependent random variables whose dependence is defined in terms of a Stein identity, \cite{Ch13S} obtained a general Cram\'er-type moderate deviation result for normal approximation using Stein's method.
Stein's method, introduced by \cite{St72B}, is a completely different approach to distribution approximation than the classical Fourier transform. It works not only for independent random variables but also for dependent random variables. It can also provide accuracy  of   the distribution approximation. Extensive applications of Stein's method to obtain Berry--Esseen-type bounds can be found in, for example, \cite{diaconis1977finite}, \cite{St86A}, \cite{Ba90S},  \cite{Go97S}, \cite{Ch04N,Ch07N}, \cite{Ch08N}, \cite{nourdin2009stein} and \cite{Sh18B}. We refer to \cite{Ch11Na}, \cite{No12N}
and \cite{Ch14S} for comprehensive coverage of the method's fundamentals and applications. In addition to the normal approximation, \cite{Ch11Nb} obtained a general nonnormal approximation via the exchangeable pair approach and the corresponding Berry--Esseen-type bounds. We also refer to \cite{Sh16I} for a more general result.

The main purpose of this paper is to obtain a Cram\'er-type moderate deviation theorem for nonnormal approximation. Our main tool is based on Stein's method, combined with some techniques in \cite{Ch11Nb} and \cite{Ch13S}. The paper is organized as follows. Section 2 presents a Cram\'er-type moderate deviation theorem under a general Stein identity setting, which recovers the result of \cite{Ch13S} as a special case. In Section 3, the result is applied to two examples: the general Curie--Weiss model and imitative monomer-dimer models.
The proofs of the main results in Section 2 are given in Sections 4 and the proofs of theorems in Section 3 are postponed to Section 5. 

\section{Main Results}\label{main}

\setcounter{equation}{0}

Let $W:=W_n$ be the random variable of interest. Following the setting in \cite{Ch11Nb} and \cite{Ch13S}, we assume that there exists a constant $\delta$, a nonnegative random function $\hatK(t)$, a function $g$ and a random variable $R(W)$ such that
\begin{equation}
	\E\bigl(f(W)g(W)\bigr)= \E\left(\int_{|t|\leq \delta} f'(W+t) \hat{K}(t) dt\right) + \E\bigl(f(W)R(W)\bigr)
  \label{c1}
\end{equation}
for all absolutely continuous functions $f$ for which the expectation of either side exists. Let
\begin{equation}
\hat{K}_1 =\inttd \hat{K}(t)dt \label{k1}
\end{equation}
and
\begin{equation}
G(y) = \int_0^y g(t) dt. \label{Gy}
\end{equation}
Let $Y$ be a random variable with the probability density function
\begin{equation}\label{y:den}
    p(y) = c_1\, e^{-G(y)}, \quad y\in \R,
\end{equation}
where $c_1$ is a normalizing constant.

In this section, we present a Cram\'er-type moderate deviation theorem for general  distribution approximation under Stein's identity in general and under an exchangeable pair and Stein's couplings in particular.

Before presenting the main theorem, we first give some of the conditions of $g$.

Assume that
\begin{condition}
  \item \label{icon1}The function $g$ is nondecreasing and $g(0)=0$.
  \item \label{icon2} For $y\neq 0,~yg(y)>0$.
  \item \label{icon3} There exists a positive constant $c_2$ such that for $x,y\in \R$,
    \begin{equation}
      |g(x+y)|\leq c_2\left( |g(x)| + |g(y)| + 1 \right).
      \label{c2}
    \end{equation}
  \item \label{icon4} There exists $c_3\geq 1$ such that for $y\in \R$,
    \begin{equation}
      |g'(y)|\leq c_3\left( \frac{1+|g(y)|}{ 1+|y|} \right).
      \label{c3}
    \end{equation}
\end{condition}

 A large class of functions satisfy \cref{icon1,icon2,icon3,icon4}. A typical example is
    $g(y)=\sgn(y)|y|^p, \ p\geq 1$.

We are now ready to present our main theorem.
\begin{theorem}\label{thm:2.1}
Let $W$ be a random variable of interest satisfying \eqref{c1}. Assume that \cref{icon1,icon2,icon3,icon4} are satisfied.
Additionally, assume that there exist $\tau_1>0,\tau_2 > 0, \delta_1>0$ and $ \delta_2\geq 0$ such that
\begin{align}
    |\E(\hat{K}_1|W)-1|&\leq \delta_1 \bigl( |g(W)|^{\tau_1}+ 1 \bigr) ,\label{de1}\\
    |R(W)|&\leq \delta_2 \bigl(|g(W)|^{\tau_2} + 1\bigr). \label{de2}
\end{align}
In addition, there exist constants $d_0\geq 1, d_1>0$ and $0 \leq  \alpha < 1$ such that
\begin{align}
\E(\hat{K_1}|W) & \leq d_0,\label{d0}\\
\delta |g(W)|  & \leq d_1 , \label{d1}\\
|R(W)|         & \leq \alpha (|g(W)| + 1). \label{d2}
\end{align}
Then, we have
\begin{equ}\label{MD}
    \frac{\P(W>z)}{ \P(Y>z)}={}& 1+O(1)\Big(\delta \bigl(1 + zg^2(z)\bigr) \\
                               & \hspace{1.7cm} +\delta_1\bigl( 1 + zg^{\tau_1+1}(z)\bigr) +  \delta_2\bigl( 1 + zg^{\tau_2}(z)\bigr) \Big)
\end{equ}
for $z\geq 0$ satisfying $\delta zg^2(z)+\delta_1zg^{\tau_1+1}(z)+\delta_2zg^{\tau_2}(z)\leq 1$, where
$O(1)$ is bounded by a finite constant depending only on $d_0,d_1,c_1,c_2, c_3, \tau_1, \tau_2, \alpha$ and $\max (g(1), |g(-1)| )$.
\end{theorem}

The condition \eqref{c1} is called a general Stein identity, see \citet*[][Chapter 2]{Ch11Na}.
We use the exchangeable pair approach and Stein's coupling to construct $\hatK(t)$ and $R(W)$ as follows.

  Let $(W,W')$ be an exchangeable pair, that is, $(W, W')$ has the same joint distribution as $(W', W)$. Let $\Delta = W - W'$. 
  Assume that
  \begin{equation}\label{expair}
  \E(\Delta|W)=\lambda(g(W)-R(W)),
  \end{equation}
  where $0 < \lambda < 1$ . Assume that $|\Delta|\leq \delta$ for some constant $\delta>0$.
  It is known (see, e.g., \citet{Ch11Nb}) that \cref{c1} is satisfied with
   $$
  \hatK(t)=\frac{1 }{ 2\lambda } \Delta(\I(-\Delta\leq t\leq 0)-\I(0<t\leq \Delta)).
  $$
  Clearly, we have
  $$
  \hatK_1 = \frac{1 }{ 2 \lambda } \Delta^2.
  $$

  For exchangeable pairs, we have the following corollary.
\begin{corollary}\label{cor:2.1}
For \((W, W')\) an exchangeable pair satisfying \eqref{expair},  assume that $g(W)$, $\hatK_1$ and $R(W)$ satisfy the \cref{icon1,icon2,icon3,icon4} and \eqref{de1}--\eqref{d2} stated in
 \cref{thm:2.1}; then, \eqref{MD} holds.
\end{corollary}

Stein's coupling introduced by \citet*{Ch10S} is another way to construct the general Stein identity.

A triple $(W,W',T)$ is called a $g$-Stein's coupling if there is a function $g$ such that
  \begin{equation}\label{stpair}
  \E(Tf(W')-Tf(W))=\E(f(W)g(W))
  \end{equation}
  for all absolutely continuous function $f$, such that the expectations on both sides exist. Assume that $|W'-W|\leq \delta$. Then, by \citet{Ch10S}, we have
  \begin{align}
  \E(f(W)g(W))=\E\Bigl(\inttd f'(W+t)\hatK(t) dt\Bigr),
  \label{eq:coupling-identity}
  \end{align}
  where
  \begin{align}
	  \label{eq:steincouplingK}
  \hatK(t)=T(\I(0\leq t\leq W'-W)-\I(W'-W\leq t<0)).
  \end{align}
  It is easy to see that
  $\hatK_1 = T(W' - W). $

The following corollary presents a moderate deviation result for Stein's coupling.

\begin{corollary}\label{cor:2.2}
	Let \((W, W', T)\) be a \(g\)-Stein's coupling satisfying \cref{stpair,eq:coupling-identity} and let \(\hatK \) be defined as in \eqref{eq:steincouplingK} and assume that $\hatK (t) \geq 0$ for \(|t| \leq \delta\). Let $g(W)$ and $\hatK_1:=T (W' - W)$ satisfy the \cref{icon1,icon2,icon3,icon4} and \cref{de1,d0,d1} stated in \cref{thm:2.1}, then \eqref{MD} holds with $\delta_2 =0$.
\end{corollary}

\begin{remark}
    For  $s \geq 0$,  let
    \begin{equ}
        \zeta (w, s) =
        \begin{cases}
            \label{r2.1-a}
            e^{G(w) - G(w- s)}, & w > s, \\
            e^{G(w)} ,          & 0 \leq w \leq s, \\
            1,                  & w < 0.
        \end{cases}
    \end{equ}
    Condition \eqref{de1} can be replaced by
    \begin{equ}
        \label{r2.1-b}
        \bigl\lvert \E(\hat{K}_1 | W) - 1 \bigr\rvert \leq K_2 + \delta_1 \bigl( |g(W)|^{\tau_1} + 1 \bigr),
    \end{equ}
    where $K_2 \geq 0$ is a random variable satisfying
    \begin{equ}
        \label{r2.1-c}
        \E K_2 \zeta(W, s) \leq \delta_1 (1 + g^{\tau_1}(s)  ) \E \zeta (W, s).
    \end{equ}
    \label{r2.1}
\end{remark}

\begin{remark}
\label{r2.2}
Condition \cref{d2} may not be satisfied when $|W|$ is large in some applications. Following the proof of \cref{thm:2.1}, when \cref{d2} is  replaced by the following condition, there exist $0 \leq \alpha < 1$, $d_2\geq 0,d_3 > 0$ and $ \kappa>0$ such that
\begin{equation}
|R(W)| \leq \alpha \,  ( |g(W)|+1) + d_2 \, I(|W|> \kappa),  \label{r2.2-a}
\end{equation}
and
\begin{equation}
    d_2  \P(|W|> \kappa ) \leq d_3 e^{ - 2 s_0 d_1^{-1} \delta^{-1} } ,
\label{r2.2-b}
\end{equation}
where $d_1$ is bounded in \cref{d1} and  $s_0 = \max \left\{ s: \delta s g^2(s) \leq 1 \right\}$,
\cref{thm:2.1,cor:2.1,cor:2.2} remain valid with
$O(1)$ bounded by a finite constant depending only on $d_0,d_1,d_2, d_3,c_1,c_2, c_3, \tau_1, \tau_2, \alpha$ and $\max (g(1), |g(-1)| )$.
\end{remark}

\section{Applications}\label{appli}

\setcounter{equation}{0}

In this section, we apply the main results to the general Curie--Weiss model at the critical temperature and the imitative monomer-dimer model.

\subsection{General Curie-Weiss model at the critical temperature}

Let $\xi$ be a random variable with probability measure $\rho$ which is   symmetric  on $\R$. Assume that
\begin{align}
    \label{eq:gcw01}
    \E \xi^2 = 1, \quad \E \exp(\beta \xi^2 / 2 ) < \infty \quad \text{for} \quad \beta \geq 0.
\end{align}
The general Curie-Weiss model $\text{CW}(\rho)$ at inverse temperature $\beta$ is defined as the array of spin random variables
$\mathbf{X} = (X_1, X_2, \ldots, X_n)$ with joint distribution
\begin{align}
    \dd \P_n (\mathbf{x}) = Z_n^{-1} \exp \biggl( \frac{\beta}{2n} (x_1 + x_2 + \cdots + x_n)^2  \biggr) \prod_{i=1}^n \dd \rho (x_i)
    \label{eq:gcw02}
\end{align}
for $\mathbf{x} = (x_1, x_2, \ldots, x_n) \in \R^n$ where
\begin{align*}
    Z_n = \int \exp\biggl( \frac{\beta}{2n} (x_1 + x_2 + \cdots + x_n)^2 \biggr) \prod_{i=1}^n \dd \rho(x_i)
\end{align*}
is the normalizing constant.

The magnetization $m(\mathbf{x})$ is defined by
\begin{align*}
    m(\mathbf{x}) = \frac{1}{n} \sum_{i = 1}^n x_i.
\end{align*}
Following the setting of \cite{Ch10A}, we assume that the measure $\rho$ satisfies the following conditions:
\begin{condition3}
\item \label{con:gcw01} $\rho$ has compact support, that is, $\rho( [-L, L] ) = 1$ for some $L < \infty$.
\item \label{con:gcw02} 
    Let
    \begin{align}
        h(s):= \frac{s^2}{2} - \log \int \exp (sx) \dd \rho(x).
        \label{eq:gcw03}
    \end{align}
    The equation $h'(s) = 0$ has a unique root at $s = 0$.
\item \label{con:gcw03} Let $k\geq 2$ be such that $h^{(i)} (0) = 0$ for $0 \leq i \leq 2k - 1$ and $h^{(2k)} (0) > 0.$
\end{condition3}
Specially, for the simple Curie--Weiss model, where $\rho = \frac{1}{2} \boldsymbol\delta_1 + \frac{1}{2} \boldsymbol\delta_{-1}$ and $\boldsymbol\delta$ is the Dirac measure,   \crefrange{con:gcw01}{con:gcw03} are satisfied with $L = 1$ and $k = 2$. For $0 < \beta < 1$, $n^{1/2} m(\mathbf{X})$ converges weakly to a Gaussian distribution, see \cite{El78T}.  Also, \cite{Ch13S} obtained the Cram\'er-type moderate deviation for this normal approximation.  When $\beta = 1$, \cite{Si734} proved that the law of $n^{1/4} m(\mathbf{X})$ converges to $\mathcal{W}(4, 12)$ as $n \to \infty$, with the probability density function
\begin{align}
    f_Y(y) = \frac{\sqrt{2}}{ 3^{1/4} \Gamma(1/4)} e^{-\frac{y^4}{12}}.
    \label{YY}
\end{align}
\cite{Ch11Nb} showed that the Berry--Esseen bound is of order $O(n^{-1/2})$.

For the rest of this subsection, we consider only the case where $\beta = 1$.
Assume that \crefrange{con:gcw01}{con:gcw03} are satisfied. Let $W = n^{\frac{1}{2k}} m(\mathbf{X})$.
\cite{El78T} showed that
$W$ converges weakly to a distribution with density
\begin{equ}
    \label{eq:gcw04}
    p(y) = c_1 \exp\bigl( - h^{(2k)}(0) y^{2k}  / (2k)!  \bigr),
\end{equ}
where $c_1$ is a normalizing constant.
For the concentration inequality,
\cite{Ch10A} used Stein's method to prove that for any $n \geq 1$ and $t \geq 0$,
\begin{align*}
    \P \bigl(  \bigl\lvert W \bigr\rvert \geq t \bigr) \leq 2 e^{-c_\rho t^{2k}},
\end{align*}
where $c_\rho > 0$ is a constant depending only on $\rho$.
Moreover, \cite{Sh18B} proved the Berry--Esseen bound:
\begin{align}
    \sup_{z \in \R} \bigl\lvert \P(W \leq z) - \P(Y \leq z) \bigr\rvert \leq C n^{-\frac{1}{2k}},
    \label{eq:gcw05}
\end{align}
where $Y \sim p(y) $ as defined in \cref{eq:gcw04} and $C>0$ is a constant.

In this subsection, we provide the Cram\'er-type moderate deviation for $W$.

\begin{theorem}
    Let $W$ be defined as above. If $\beta = 1$, we have
    \begin{align*}
        \frac{ \P(W > z)  }{\P(Y > z)} = 1 + O(1) n^{-1/k} ( 1 + z^{2k + 2} ),
    \end{align*}
    uniformly in $z \in \bigl(0, n^{\frac{1}{k(2k + 2)}}\bigr)$.
    \label{thm:3.2}
\end{theorem}
\begin{corollary}
    For the simple Curie--Weiss model, in which case $\rho = \frac{1}{2} \boldsymbol\delta_1 + \frac{1}{2} \boldsymbol\delta_{-1}$ and $\boldsymbol\delta$ is the Dirac measure. Then,
    \begin{align*}
        \frac{ \P(W > z)  }{\P(Y > z)} = 1 + O(1) n^{-1/2} ( 1 + z^{6} ),
    \end{align*}
    uniformly in $z \in (0, n^{1/12})$, where $Y \sim \mathcal{W}(4,  12)$.
    \label{cor:3.1}
\end{corollary}
After we finished this paper, we learnt that \cite{Ca17C} proved \cref{cor:3.1} by a completely different approach.

\begin{remark}
    Comparing to \cite[][Theorem 3.2, (ii)]{Sh18B}, we assume the additional condition that $\rho$ is a symmetric measure.
    Following the proofs of \cref{thm:3.2} and \cite[][Theorem 3.2]{Sh18B}, we have \cref{eq:gcw05} can be improved to
    \begin{align*}
        \sup_{z \in \R}\bigl\lvert \P(W\leq z) - \P(Y\leq z) \bigr\rvert \leq C n^{-1/k}.
    \end{align*}
    \label{rem:3.1}
\end{remark}

\subsection{The imitative monomer-dimer mean-field model}\label{sec:imd}
In this subsection, we consider the imitative monomer-dimer model and give the moderate deviation result. A pure monomer-dimer model can be used to study the properties of diatomic oxygen molecules deposited on tungsten or liquid mixtures with molecules of unequal size, see \citep{Fo37A,Ro38S} for
example.  \citet*{Ch39S} studied the attractive component of the van der Waals potential, while \citet*{Al15L} and \cite{Al14M} considered the asymptotic properties.

\citet*{Ch16L} recently obtained the Berry--Esseen bound by using Stein's method.
In this subsection, we apply our main theorem to obtain the moderate deviation result.

 For $n\geq 1$, let $G=(V,E)$ be a complete graph with vertex set $V=\{1,\dots,n\}$ and edge set $E=\bigl\{uv=\{u,v\}:u,v\in V, u<v \bigr\}$. A dimer configuration on the graph $G$ is a set $D$ of pairwise nonincident edges satisfying the following rule: if $uv\in D$, then for all $w\neq v$, $uw\not\in D$. Given a dimer configuration $D$, the set of monomers $\M(D)$ is the collection of dimer-free vertices.
Let $\D$ denote the set of all dimer configurations. Denote the number of elements by $\#(\cdot)$. Then, we have
\[
2\#(D) + \#(\M(D)) = n.
\]

We now introduce the imitative monomer-dimer model. The Hamiltonian of the model with an imitation coefficient $J\geq 0$ and an external field $h\in \R$ is given by
\[
    -T(D) = n(J m(D)^2 + b m(D))
\]
for all $D\in \D$, where $m(D) = \#(\M(D))/n$ is called the monomer density and the parameter $b$ is given by
\[
 b = \frac{\log n}{ 2}+ h - J .
\]
The associated Gibbs measure is defined as
\[
    p(D) = \frac{e^{-T(D)}}{ \sum_{D\in \D}  e^{-T(D)}}.
\]

Let
\begin{align}
H(x) = -Jx^2 - \frac{1}{ 2} \left( 1 - g \bigl(\tau (x) \bigr)+ \log \bigl(1-g( \tau(x))\bigr) \right),
\label{eq:Hg}
\end{align}
where
\[
g(x) = \frac{1}{ 2 } \left( \sqrt{e^{4x} + 4 e^{2x}} - e^{2x} \right), \quad \tau(x) = (2x-1) J + h.
\]
\citet*{Al14M} showed that the imitative monomer-dimer model exhibits the following three phases.
Let $$
J_c = \frac{1}{ 4(3-2\sqrt{2})} , \quad h_c = \frac{1}{ 2} \log(2\sqrt{2} -2 ) - \frac{1}{ 4}.
$$
There exists a function $\gamma : (J_c,\infty) \rightarrow \R$ with $\gamma (J_c) = h_c$ such that
if $(J,h)\not\in\Gamma$, where $\Gamma := \{(J, \gamma(J)):J>J_c\}$,
then the function $ H(x)$
has a unique maximizer $m_0$ that satisfies
$m_0 = g(\tau(m_0))$.
Moreover, if $(J,h) \not\in \Gamma \cup \{ (J_c,h_c) \}$, then $H''(m_0)<0$. If $(J,h) =(J_c,h_c)$, then $m_0=m_c:= 2 - \sqrt{2}$ and
\[
H'(m_c)= H''(m_c)=H^{(3)}(m_c)=0,
\]
but
\[
H^{(4)}(m_c) <0.
\]

If $(J,h)\in \Gamma$, then $H(s)$
has two distinct maximizers; therefore, in this case, $m(D)$ may not converge. Hence, we consider only the cases when $(J,h)\not\in \Gamma$.

\citet*{Al14M} showed that when $(J,h)\not\in\Gamma \cup \{(J_c,h_c)\}$, $n^{1/2}(m(D)-m_0)$ converges to a normal distribution with zero mean and variance $\lambda_0 = -(H''(m_0))^{-1} - (2J)^{-1}$. However, when $(J,h)=(J_c,h_c)$, $n^{1/4}(m(D) - m_0)$ converges to $Y$ in distribution, whose p.d.f. is given by
\begin{equation}\label{Yden}
p(y) = c_1 e^{-\lambda_c y^4/24}
\end{equation}
with $\lambda_c = - H^{(4)}(m_c) > 0$ and $c_1$ is a normalizing constant. \citet*{Ch16L} obtained the Berry--Esseen bound using Stein's method.

We use the following notations.
Let $\Sigma = \{0,1\}^n$. For each $\sigma=(\sigma_1, \dots ,  \sigma_n) \in \Sigma,$ define a Hamiltonian
\[
-T(\sigma) = n(J m(\sigma)^2 + b m(\sigma)),
\]
where $m(\sigma) = n^{-1}(\sigma_1+ \cdots + \sigma_n)$ is the magnetization of the configuration $\sigma$. Denote by $\A(\sigma)$ the set of all sites $i \in V$ such that $\sigma_i = 1$. Also, let $D(\sigma)$ denote the total number of dimer configurations $D\in \D$ with $\M(D) = \A(\sigma)$. Therefore, the Gibbs measure can be written as
\[
p(\sigma) = \frac{D(\sigma) \exp(-T(\sigma)) }{ \sum_{\tau\in \Sigma} D(\tau) \exp(-T(\tau))}.
\]

The following result gives a Cram\'er-type moderate deviation for the magnetization.

\begin{theorem}\label{thm:3.3}
    If $(J,h) \not\in\Gamma\cup \{J_c,h_c\}$, then, for $0 \leq z\leq n^{1/6}$,
  \begin{equation}\lbl{res:imdnormal}
    \frac{\P(n^{1/2}(m(\sigma)-m_0)> z) }{ \P(Z_0>z)} = 1 + O(1)n^{-1/2}(1+z^3),
  \end{equation}
  where $Z_0 $ follows normal distribution with zero mean and variance $\lambda_0 = -(H''(m_0))^{-1} - (2J)^{-1}$. If $(J,h)=(J_c,h_c)$, then
  for $0 \leq z \leq n^{1/20},$
  \begin{equation}\lbl{res:imdnonnormal}
    \frac{\P(n^{1/4}(m(\sigma)-m_c)> z) }{  \P(Y> z)} = 1+ O(1) n^{-1/4}(1+z^5),
  \end{equation}
  where $Y$ is a random variable with the probability density function given in  \eqref{Yden}.
\end{theorem}

\section{Proofs of main results}\label{proof}

\setcounter{equation}{0}

In this section, we give the proofs of the main theorems. 
In what follows, we use $C$ or $C_1, C_2, \ldots$ to denote a finite constant depending only on  $c_1, c_2, c_3, d_0, d_1, \tau_1, \tau_2, \mu_1$ and  $\alpha$, where $\mu_1 =  \max(g(1), |g(-1)|)  + 1$, and \(C\) might be different in different places.

\subsection{Proof of Theorem 2.1}\label{sub:proofmain}

Let $Y$ be a random variable with a probability density function given in \eqref{y:den}
and $F(z)$ be the distribution function of $Y$. We start with a
preliminary lemma on the properties of $(1 - F(w)) / p(w)$ and $F(w)/p(w)$, whose proof is postponed to \cref{sub:lem4}.

\begin{lemma}\label{lem:4.1}
  Assume that \cref{icon1,icon2,icon3,icon4} are satisfied. Then, we have
\begin{equation}\label{wg0}
    \frac{1}{ \max(1, c_3) (1+g(w))} \leq
\frac{1-F(w)}{ p(w)}\leq
\min\biggl\{\frac{1}{ g(w)},1/c_1 \biggr\} \ \ \mbox{for} \  \ w >0
 \end{equation}
  and
\begin{equation}\label{wg<0}
    \frac{F(w)}{ p(w)}\leq \min\left\{\frac{1}{ |g(w)|},1/c_1\right\} \ \ \mbox{for} \ \ w <0.
\end{equation}

\end{lemma}


Let $f_z$ be the solution to Stein's equation
\begin{equation}
\label{ste:eq}
 f'(w)-f(w)g(w)=\I(w\leq z)-F(z).
\end{equation}
As shown in \citet*{Ch11Nb}, the solution $f_z$ can be written as
\begin{equation}\label{eq:solution}
f_z(w)=\begin{cases}
\dfrac{F(w)(1-F(z))}{p(w)},& w\leq z;\\[5mm]
\dfrac{F(z)(1-F(w))}{p(w)},& w> z.
\end{cases}
\end{equation}

Let
\begin{equ}
	\label{eq:I1234}
I_1&= \E\Bigl( \inttd \big| f_z(W+t)g(W+t)-f_z(W)g(W)  \big|\hatK (t)dt\Bigr),\\
I_2&=\E(|(\E(\hatK_1|W)-1)f_z(W)g(W)|), \\
I_3&=\E\left(|(\E(\hatK_1|W) -1)(P(Y>z)-\I(W>z+\delta))|\right), \\
I_4&=\E(f_z(W)|R(W)|).
\end{equ}
The following propositions provide estimates of $I_1, I_2, I_3$ and $I_4$, whose proofs are given in Subsection 4.4.

\begin{proposition}
\label{pro:4.1}
If $\delta \leq 1$, then
  \begin{equation}
      I_1 \leq  C \delta.
    \label{IA}
  \end{equation}

  Assume that $z\geq 0, \max(\delta, \delta_1, \delta_2) \leq 1$ and $\delta z g^2(z) + \delta_1 z g^{\tau_1 + 1} (z) + \delta_2 z g^{\tau_2} (z) \leq 1$. 
  Then, we have
  \begin{equation}
    I_1\leq C\delta (1+zg^2(z)) (1-F(z)).
    \lbl{IB}
  \end{equation}
  \end{proposition}

\begin{proposition}
\label{pro:4.2}
We have
\begin{align}
    I_2 + I_3 \leq C \delta_1, \quad I_4 \leq C \delta_2.
    \label{pro:4.2-00}
\end{align}
For $z > 0, \max(\delta, \delta_1, \delta_2) \leq 1$ and $\delta z g^2(z) + \delta_1 z g^{\tau_1 + 1} (z) + \delta_2 z g^{\tau_2} (z) \leq 1$, we have
\begin{align}
    I_2 + I_3 & \leq C   \delta_1 \bigl( 1 + z g^{\tau_1 + 1}(z) \bigr)   (1 - F(z)),
    \label{pro:4.2-01}\\
    I_4 & \leq C \delta_2 \bigl( 1 + z g^{\tau_2}(z) \bigr)   (1 - F(z)).
    \label{pro:4.2-02}
\end{align}

\end{proposition}

We are ready to give the proof of \cref{thm:2.1}.
\begin{proof}[Proof of \cref{thm:2.1}]
From \eqref{c1}, we have
\begin{equ}
\ML {\E(f_z(W)g(W)-f_z(W)R(W))}\nn\\
={}& \E\left(\inttd f'_z(W+t)\hat{K}(t)dt\right)\nn\\
={}& \E\Bigl(\inttd \big( f_z(W+t)g(W+t) \\
   & \hspace{2cm} +\P(Y>z)-\I(W+t>z)   \big)\hatK (t)dt\Bigr)\nn\\
\leq {}& \E \Bigl(\inttd \big( f_z(W+t)g(W+t)-f_z(W)g(W)  \big)\hatK (t)dt\Bigr) \\
       & +\E(\hatK_1 f_z(W)g(W) )\nn\\
& + \E\left(\hatK_1\big( \P(Y>z)-\I(W>z+\delta) \big)\right)\nn\\
\leq {}&  \E\left( \inttd \big| f_z(W+t)g(W+t)-f_z(W)g(W)  \big|\hatK (t)dt\right) \\
       & +\E(\hatK_1 f_z(W)g(W) )\nn\\
& + \E\left(|\E(\hatK_1 | W) -1|\big| P(Y>z)-\I(W>z+\delta) \big|\right) \\
& +\P(Y>z)-\P(W>z+\delta). \label{in:e1}
\end{equ}
Rearranging \eqref{in:e1} leads to
\begin{equ}
 {\P(W>z+\delta)-\P(Y>z)} \leq I_1 + I_2 + I_3 + I_4,
 \label{in:e2}
\end{equ}
where \(I_1, I_2, I_3\) and \(I_4\) are defined as in \cref{eq:I1234}.

First, we use \eqref{in:e2} and \cref{pro:4.1,pro:4.2} to prove the Berry--Esseen bound
\begin{align}
       {| \P(W>z)-\P(Y>z)| }\leq   C (\delta + \delta_1 + \delta_2), 
\lbl{BS-a}
\end{align}
where $C \geq  1$. 
By \cref{in:e2,IA,pro:4.2-00}, for $\delta \leq 1$, we have
\begin{align}
    {\P(W>z+\delta)-P(Y>z)}  \leq C(\delta+ \delta_1 + \delta_2). \label{bon:I}
\end{align}
Together with
\[
    \P(Y>z)-\P(Y>z+\delta)\leq c_1\int_z^{z+\delta} e^{-G(w)}dw\leq c_1\delta,
\]
we have
$$
\P( W > z) - \P(Y >z) \leq C ( \delta + \delta_1 + \delta_2).
$$
Similarly, we have
$$
\P( W > z) - \P(Y >z) \geq -  C ( \delta + \delta_1 + \delta_2).
$$
This proves the inequality \eqref{BS-a} for $\delta \leq 1$. For $\delta > 1$, \cref{BS-a} is trivial because $C \geq 1$.

 Next, we move to prove \eqref{MD}. Let  $z_0>1$ be a constant such that
\[
\min\bigl\{ z_0g^2(z_0), z_0g^{\tau_1+1}(z_0), z_0g^{\tau_2}(z_0), z_0  \bigr\}\geq 1.
\]

For $0 \leq z \leq z_0$, \eqref{MD} follows from \eqref{BS-a} because
\begin{align}
\frac{\P(W>z)-\P(Y>z)}{ \P(Y>z)}\leq \frac{C (\delta+\delta_1+\delta_2)}{ 1-F(z_0)}, \nn
\end{align}
where $C$ is a constant.

For $z > z_0$, and thus \(z >1\),  we can assume
$\max\{\delta,\delta_1,\delta_2  \}\leq 1$; otherwise, it would contradict the condition
\begin{align}
	\label{eq:conditionzzz}
	\delta zg^2(z)+\delta_1zg^{\tau_1+1}(z)+ \delta_2zg^{\tau_2}(z)\leq 1.
\end{align}

In this case, it follows that
  \begin{equ}
	  \label{eq:delta-gz}
	  \delta \leq 1, \quad \delta g^2(z) \leq \delta z g^2(z) \leq 1,
  \end{equ}
  provided that \cref{eq:conditionzzz} holds.

  By \eqref{in:e2} and \cref{pro:4.1,pro:4.2},
\begin{equ}\label{del}
    \ML {\P(W> z+\delta )-\bigl(1-F(z)\bigr)}\nn\\
	\leq {}& I_1 + I_2 +I_3 + I_4 \\
    \leq {}&   C (1-F(z))  \Bigl( \delta (1 + zg^2(z)) \\
           & \hspace{2.5cm} + \delta_1 (1 + zg^{\tau_1 + 1} (z)) + \delta_2 (1 + zg^{\tau_2} (z))\Bigr)  .
\end{equ}
By replacing $z$ with $z-\delta,$ and noting that $g$ is nondecreasing, we can rewrite \eqref{del}  as
\begin{equ}\label{delrep}
\ML {\P(W> z )-(1-F(z -\delta ))}\nn\\  \leq {}&  C (1-F(z-\delta )) \Bigl( \delta (1 + zg^2(z)) + \delta_1 (1 + zg^{\tau_1 + 1} (z))\\
                                               &  \phantom{ C (1-F(z-\delta )) \Bigl( \delta} + \delta_2 (1 + zg^{\tau_2} (z))\Bigr)  .
\end{equ}

As $p(y)$ is decreasing in $[z-\delta,z]$, we have
\begin{align*}
F(z)-F(z-\delta)&= \int_{z-\delta}^zp(t)dt\nn\\
&\leq \delta p(z-\delta)
\leq  e^{\delta g(z)} \delta p(z).
\end{align*}

By \cref{eq:delta-gz}, it follows that $\delta g(z)\leq (1/2) \delta (1 + g^2(z)) \leq 1$. By \eqref{wg0}, we also have
\[
    p(z)\leq \max(1, c_3) (1+g(z))(1-F(z));
\]
then,
\[
F(z)-F(z-\delta)\leq C\delta(1+g(z))(1-F(z))
\]
for some constant $C$.
Recall that $\delta(1+g(z))\leq 2$; then,
\[
1-F(z-\delta)\leq C(1-F(z)).
\]
Together with \eqref{delrep}, we get
\begin{align*}
\ML {\P(W> z )-(1-F(z))}\\  \leq {}&
\P(W> z) - (1-F(z-\delta)) + F(z) - F(z-\delta) \\
  \leq {}&
C (1-F(z-\delta ))\Bigl( \delta (1 + zg^2(z)) + \delta_1 (1 + zg^{\tau_1 + 1} (z)) + \delta_2 (1 + zg^{\tau_2} (z))\Bigr)  \\
{}&  + C\delta(1+g(z))(1-F(z))\\
 \leq {}&  C (1-F(z))\Bigl( \delta (1 + zg^2(z)) + \delta_1 (1 + zg^{\tau_1 + 1} (z)) + \delta_2 (1 + zg^{\tau_2} (z))\Bigr)  .
\end{align*}
Similarly, we can prove the lower bound as follows:
\begin{align*}
\ML {\P(W> z )-(1-F(z))}\\  \geq {}&
- C (1-F(z))\Bigl( \delta (1 + zg^2(z)) + \delta_1 (1 + zg^{\tau_1 + 1} (z)) + \delta_2 (1 + zg^{\tau_2} (z))\Bigr)  .
\end{align*}
This completes the proof of \cref{thm:2.1}.
\end{proof}


\subsection{Proof of Lemma 4.1}
\label{sub:lem4}
For $w\geq 0$, by the monotonicity of $g(\cdot)$,
we have
	\begin{align*}
	1-F(w)={}&\int_w^\infty p(t) dt\nn\\
	={}& c_1\int_w^\infty e^{-G(t)}dt\nn\\
	={}& c_1\int_w^\infty
    \frac{1}{ g(t)} e^{-G(t)} dG(t) \nn\\
	\leq {}&
    \frac{c_1}{ g(w)} e^{-G(w)}\nn\\
    ={}&
    \frac{p(w)}{ g(w)}.\nn
	\end{align*}
Let $H(w)=1-F(w)-p(w)/c_1$; then,
\begin{align*}
    H'(w) = p(w) (g(w)/c_1 - 1).
\end{align*}
Note that $g(w)/c_1 = 1$ has at most one solution in $(0, +\infty)$ and that $g(0) = 0$; then, $H(w)$ takes the maximum at either $0$ or $+\infty$. We have
\begin{align*}
    H(w) \leq \max \bigl\{H(0), \lim_{w\to \infty} H(w)\bigr\} \leq 0.
\end{align*}
This proves the upper bound of \cref{wg0}. The inequality \eqref{wg<0} can be obtained similarly.

To finish the proof, we need to prove that for $w\geq 0$,
\begin{equation}\lbl{c211}
	\frac{p(w) }{ 1+g(w)} \leq \max (1, c_3) (1-F(w)).
\end{equation}
Let
\begin{equation}
	\zeta(w)=\frac{1}{ 1+g(w)}e^{-G(w)}.\nn
	\end{equation}
	As $g'(w)\leq c_3(1+g(w)),$ we have
	\[
	- \zeta'(w)= \frac{g(w)}{ 1+g(w)} e^{-G(w)}
    +\frac{g'(w)}{ (1+g(w))^2} e^{-G(w)} \leq \max (1, c_3)  e^{-G(w)}.
	\]
    As $g(w)$ is nondecreasing and $g(w ) > 0$ for $w > 0$, then $G(w) = \int_0^w g(t) dt \to \infty$ as $w \to \infty$. Therefore,  $\lim_{w \to \infty} p(w) = 0$.
    Taking the integration on both sides yields
    \begin{align*}
        \zeta(w) =  - \int_w^\infty \zeta'(t) dt  \leq \max( 1, c_3)  \int_w^\infty e^{-G(t) } dt,
    \end{align*}
    which leads to \cref{c211}.
    This completes the proof.

\subsection{Preliminary lemmas}%
\label{sub:preliminary_lemmas}

To prove \cref{pro:4.1,pro:4.2}, we first present some preliminary lemmas. Throughout this subsection, we assume that conditions (A1)–(A4) are satisfied.

\begin{lemma}\label{lem:4.2}
Assume that $ 0 < \delta \leq 1$. Then, we have
\begin{equation}\label{bon2}
    \sup_{|t|\leq \delta}|g(w+t)|\leq c_2(|g(w)|+ \mu_1 ),
\end{equation}
where $ \mu_1 = \max(g(1), |g(-1)|)  +1$.

Also, for $w>s> 0$ and any positive number $a>1$, there exists $b(a)$ depending on $a$, $c_2$ and $c_3$,  such that
\begin{equation}\label{bon4}
g(w)-g(w-s)\leq \frac{1}{ a}\, g(w)+b(a)(g(s)+1),
\end{equation}
where one can choose
\[b(a) = \bigl(  (2c_2) + \cdots + (2c_2)^{m(a)} \bigr)  + 1/a,
\]
and  $m(a) = [\log_2 (ac_3 + 1)] + 1$. 
\end{lemma}

\begin{proof}[Proof of \cref{lem:4.2}]
The inequality \cref{bon2} can be derived immediately from \eqref{c2}.
Meanwhile, \eqref{bon4} remains to be shown. 
For $a>1$, consider two cases.
\begin{description}[leftmargin=0pt]
    \item[{\it Case 1.}]  If $s < w \leq (ac_3 + 1)s$, denote $m := m(a) = [\log_2 (ac_3 + 1)] + 1$.  As $g$ is nondecreasing and by \cref{c2}, we have
        \begin{align*}
            g (w ) & \leq g\bigl( 2^m s \bigr)  \leq 2 c_2 g \bigl(  2^{m-1} s \bigr) + c_2 .
        \end{align*}
        By induction, we have
        \begin{equ}
            g(w) & \leq (2c_2)^m g(s) + c_2 ( 1 + (2c_2) + \cdots + (2c_2)^{m-1} ) \\
                 & \leq b(a) (g(s) + 1),
            \label{l41-01}
        \end{equ}
        where $b(a) =2 c_2 ( 1 + (2c_2) + \cdots + (2c_2)^{m(a) -1} ) + 1/a $.
    \item [{\it Case 2.}] If $w > (ac_3 + 1)s$, by \cref{c3}, we have
        \begin{equ}
            g(w) - g(w - s) & = \int_0^{s} g'(w - t) dt \\
                            & \leq c_3 \int_0^s \frac{ 1 + g(w - t) }{ 1 + (w - t) } dt\\
                            & \leq \frac{1}{a} (g(w) + 1).
            \label{l41-02}
        \end{equ}
\end{description}
By \cref{l41-01,l41-02}, this completes the proof.
\end{proof}

\begin{lemma}\label{lem:4.3}
  For $w\geq 0$ and any $a>0$, we have 
  \begin{equation}\label{eq:g'w}
    g'(w)\leq \frac{1}{ a} g(w) +  c_3(g(a c_3 ) + 1) + 1/a.
  \end{equation}
\end{lemma}
\begin{proof}[Proof of \cref{lem:4.3}]
  Recall that \eqref{c3} states that for $w\geq 0$,
  \[
  g'(w)\leq c_3\left( \frac{1+g(w) }{ 1 + w} \right).
  \]
  Fix $a>0$. When $w>ac_3$, we have
  \[
  g'(w)\leq \frac{1}{ a} (g(w) + 1).
  \]
  When $w\leq ac_3$, by the monotonicity property of $g$, we have
  \[
  g'(w) \leq c_3(g(a c_3) + 1).
  \]
 This completes the proof.
\end{proof}

For $s>0$, define
\begin{equation}
	\label{Fws}
	f(w,s)=
	\begin{cases}
		e^{G(w)-G(w-s)}-1,&w>s,\\
		e^{G(w)}-1,& 0 \leq w\leq s,\\
		0,& w\leq 0.
	\end{cases}
\end{equation}

We next consider a ratio property of $f(w,s)$.
It is easy to see that $f(w,s)$ is absolutely continuous with respect to both $w$ and $s$, and the partial derivatives are
\begin{equ}
    \frac{\partial }{ \partial w} f(w,s) & = e^{G(w)-G(w-s)}(g(w)-g(w-s))\I(w>s) \\
                                         & \phantom{=} \;  + e^{G(w)}g(w)\I(0\leq w\leq s)
    \label{eq:fpartialw}
\end{equ}
and
\begin{equ}
    \frac{\partial }{ \partial s} f(w,s) = e^{G(w)-G(w-s)} g(w-s)\I(0< s\leq w).
    \label{eq:fpartials}
\end{equ}
\begin{lemma}\label{lemma:1}
    Let $f(w) := f(w,s)$ be defined as in \eqref{Fws}. For  $0 \leq \delta \leq 1$ and $\delta |g(w) | \leq d_1$,
we have
\begin{equation}\label{f:ratio}
    \sup_{|u|\leq \delta} \Big|  \frac{f(w+u)+1}{ f(w)+1} \Big| \I(w + u \geq 0) \leq \mu_2 ,
\end{equation}
where $\mu_2 =\expb{c_2 (d_1 + \mu_1) +  \mu_1 }.$
Moreover, we have
\begin{align}
    \sup_{|u| \leq \delta} |f''(w + u) | \leq \mu_3 (g^2(w) + 1) (f(w) + 1).
    \label{f''}
\end{align}
where $\mu_3 = 2 c_2^2 (c_3 + 1) (\mu_1^2 + 1) \mu_2$.
\end{lemma}
\begin{proof}
    Recall that $\mu_1 = \max(g(1) , |g(-1)|) + 1$.
    When $w + u\geq 0$ and $w \geq 0$, as $g$ is nondecreasing, we have
\begin{align*}
     \sup_{|u| \leq \delta }  \Big|  \frac{f(w+u)+1}{ f(w)+1} \Big|
    & \leq e^{G(w+\delta)-G(w)}\\
    & \leq e^{\delta|g(w+\delta)|} \leq e^{c_2(d_1 + \mu_1) },\nn
\end{align*}
where in the last inequality we use \eqref{bon2}.
When $w + u \geq 0$, $w < 0$ and $|u| \leq \delta$, we have $0 \leq w + u < \delta\leq 1$; hence, by the nondecreasing property of $g$,
\begin{align*}
      \sup_{|u| \leq \delta }  \Big|  \frac{f(w+u)+1}{ f(w)+1} \Big|
     & \leq \sup_{|u|\leq \delta} e^{G(w + u)}   \leq e^{G(\delta)} \leq e^{\mu_1}.
\end{align*}
This proves \cref{f:ratio}.

For $f''(w)$, by \cref{eq:fpartialw},
\begin{align*}
    f''(w) & = e^{G(w) - G(w - s)} \bigl( g(w) - g(w - s)\bigr)^2 \I(w > s) \\
           & \quad + e^{G(w) - G(w - s)} (g'(w) - g'(w - s)) \I(w > s)\\
           & \quad + e^{G(w)} g^2(w) \I(0 \leq w \leq s) \\
           & \quad + e^{G(w)} g'(w) \I(0 \leq w \leq s).
\end{align*}
As $g$ is nondecreasing, we have $g'(w - s) \geq 0$; thus, $g'(w)- g'(w - s) \leq g'(w)$. For $w > s$,
$0 \leq g(w) - g(w - s) \leq g(w)$. Therefore,
\begin{align*}
    f''(w) & \leq \bigl(g'(w) + g^2(w)\bigr) \bigl(f(w) + 1\bigr) \I(w \geq 0).
\end{align*}
By \cref{c3}, for $c_3 > 1$, we have
\begin{align*}
    g^2(w) + g'(w) \leq g^2(w) + c_3 \bigl(1 + g(w)\bigr) \leq 2 (c_3 + 1) \bigl( g^2(w) + 1 \bigr).
\end{align*}
Hence,
\begin{align*}
    f''(w) & \leq 2 (c_3 + 1) (g^2(w) + 1) (f(w) + 1).
\end{align*}
By \cref{bon2} and \cref{f:ratio}, we have
\begin{align*}
    \sup_{|u| \leq \delta} |f''(w + u) | \leq \mu_3 (g^2(w) + 1) (f(w) + 1),
\end{align*}
where $\mu_3 = 2 c_2^2 (c_3 + 1) (\mu_1^2 + 1) \mu_2$.
This completes the proof of \cref{lemma:1}.
\end{proof}

Let $W$ be the random variable defined as in \cref{thm:2.1}. 
For $0 \leq \tau \leq \max(2, \tau_1 + 1, \tau_2)$ and $s > 0$, \cref{lem:4.4,lem:4.5} give the properties of $\E |g(W) |^{\tau}$, $\E |g(W)|^{\tau} e^{G(W)} \I(0 \leq W \leq s) $ and $\E |g(W)|^{\tau} e^{G(W) - G(W - s)}\I(W > s)$, which play a key role in the proofs of \cref{pro:4.1,pro:4.2}.

\begin{lemma}\label{lem:4.4}
    Suppose that \crefrange{icon1}{icon4} and \cref{d0,d1,d2} are satisfied with \(\delta \leq 1\).
    For $0 \leq \tau \leq \max(2, \tau_1 + 1, \tau_2)$, we have
    \begin{align}
        \E |g(W)|^{\tau} \leq C.
        \label{eq:re0}
    \end{align}
    Moreover, for $s>0$,
we have
\begin{equation}
\label{eq:re1}
\E\left(e^{G(W)-G(W-s)}g^{\tau}(W)\I(W>s)\right)\leq C(1+g^{\tau}(s))(\E(f(W,s))+1),
\end{equation}
and
\begin{equation}
\label{eq:re2}
\E\left(e^{G(W)}g^{\tau}(W)\I(0\leq W\leq s)\right)\leq C(1+g^{\tau}(s))(\E(f(W,s))+1).
\end{equation}

\end{lemma}

\begin{proof}[Proof of \cref{lem:4.4}]
In this proof, we always assume that \(\delta \leq 1\).

    We first prove \cref{eq:re0}. Without loss of generality, we consider only the case where $\tau\geq 2$. As $\delta |g(W)| \leq d_1,$ we have $\E |g(W)|^{\tau} < \infty$.
 To bound $\E |g(W)|^{\tau}$, without loss of generality, we consider only $\E g^{\tau}(W) \I(W \geq 0)$. Let $g_+(w) := g(w) \I(w \geq 0)$. As $g(0) = 0$ and $g$ is differentiable, we find that $g_+(w) $ is absolutely continuous. By \cref{c1}, we have
\begin{align}
	 \E \{g^{\tau} (W) \I(W \geq 0)\}  &\phantom{:} = \E \{g(W) \cdot g_+^{\tau- 1} (W)\} \nonumber \\
                                 & := Q_1 + Q_2,
    \label{l2-0g}
\end{align}
where
\begin{align*}
 Q_1 & = (\tau - 1) \E \intud g_+^{\tau - 2} (W + u) g'(W + u) \I(W + u \geq 0) \hatK (u) du,  \\
 Q_2 & = \E R(W) g_+^{\tau - 1} (W) .
\end{align*}
The following inequality is well known:
for any $a> 0, x, y \geq 0$ and $\tau >1$
\begin{align}
    x^{\tau - 1}y \leq \frac{\tau - 1}{a \tau} x^{\tau} + \frac{a^{\tau -1} }{\tau}y^{\tau} .
    \label{l2-01}
\end{align}
For the first term $Q_1$, by \cref{c3}, we have
\begin{align*}
    g'(w + u) \leq c_3 \bigl(1 + |g(w + u)|\bigr).
\end{align*}
Thus, for $w + u \geq 0$,
\begin{align*}
    \ML \sup_{|u| \leq \delta} g_+^{\tau - 2} (w + u) g'(w + u) \\
    & \leq c_3 \sup_{|u| \leq \delta} \bigl( g_+^{\tau - 1}(w + u) + g_+^{\tau - 2}(w + u)\bigr) \\
    & \leq 2 c_3 \sup_{|u| \leq \delta} \bigl( g_+^{\tau - 1} (w + u) + 1\bigr) \\
    & \leq \frac{1 - \alpha}{ 8 \times (2c_2)^{\tau} d_0 (\tau - 1) } \sup_{|u| \leq \delta}  |g(w + u)|^{\tau}  + D_{1, 0},
\end{align*}
where we use \cref{l2-01} with
\begin{align*}
	a = \frac{8 \times (2c_2)^{\tau +1} d_0 (\tau - 1)}{1 - \alpha} \text{ and  } x = |g_+(w + u)|
\end{align*}
in the last inequality. Here and in the sequel,  $D_{1, 0}$, $D_{2,0}$, etc. denote constants depending on $c_2, c_3, d_0, d_1, \mu_1, \alpha$ and $\tau$.
By \cref{bon2}, we have
\begin{align*}
    \sup_{|u| \leq \delta} |g(w + u) |^{\tau } \leq (2c_2)^\tau ( |g(w) |^{\tau} + \mu_1^\tau ).
\end{align*}
Then, by \cref{d0}, we have
\begin{align}
    Q_1 & \leq \frac{1 - \alpha}{8} \E |g(W)|^{\tau} + D_{2, 0}.
    \label{l2-q1}
\end{align}

For $Q_2$, by \cref{d2} and using \cref{l2-01} again, we have
\begin{align}
    Q_2 \leq \alpha \E g_+^{\tau} (W) + \frac{1 - \alpha}{4} \, \E g_{+}^{\tau} (W) +  \Big( \frac{4}{ 1 - \alpha }\Big)^{\tau -1}.
    \label{l2-q2}
\end{align}
Hence, by \cref{l2-0g,l2-q1,l2-q2}, we have
\begin{align*}
    \E g_+^\tau (W) \leq \frac{1}{6} \E|g(W)|^\tau + D_{3,0}.
\end{align*}
Similarly, we have
\begin{align*}
    \E g_{-}^\tau (W) \leq \frac{1}{6} \E |g(W) | ^{\tau} + D_{4,0}.
\end{align*}
Combining the two foregoing inequalities yields \cref{eq:re0}.

 As to \cref{eq:re1,eq:re2}, we first consider the case where $\tau \geq 2$.
  Write $f(w):=f(w,s)$.
By \eqref{c1} and \eqref{eq:fpartialw}, we have
\begin{equ}
     {\E \bigl(g(W)^{\tau} f(W) \bigr)}
	 & = {\E g(W) \bigl\{g(W)^{\tau-1} f(W) \bigr\}} \\
    &= M_1+M_2+M_3+M_4,
    \label{l4.4-0a}
\end{equ}
where
\begin{equ}
  M_1   & = \E \intud g^{\tau } (W + u) e^{G(W + u)} \I(0 \leq W + u\leq s) \hatK (u) du, \\
  M_2   & =\E \intud g^{\tau - 1} (W + u) \bigl( g(W + u) - g(W + u - s)\bigr) \\
     & \qquad \hspace{1.5cm}\times e^{G(W + u) - G(W + u - s)} \I(W + u > s)\hatK (u) du ,\\
  M_3   & = (\tau - 1)\E \intud g^{\tau - 2} (W + u) g'(W + u) f(W + u)  \hatK(u) du,  \\  
 M_4   & = \E R(W) {g}^{\tau - 1} (W) f(W) .
 \label{l4.4-0m4}
\end{equ}
We next give the bounds of $M_1, M_2, M_3$ and $M_4$.
  For $M_1$, by \cref{d0,f:ratio} and noting that $g$ is nondecreasing, we have
        \begin{equ}
            M_1 & \leq d_0 g^{\tau} (s) \E  \sup_{|u| \leq \delta} (f(W + u) + 1) \I(0 \leq W + u \leq s) \\
                & \leq d_0 \mu_2 g^{\tau} (s) \E (f(W ) + 1).
         \label{l4.4-00}
        \end{equ}
  To bound  $M_2$, we first give the bound of $g(w + u)$ and $g(w + u) - g(w + u - s)$ for $|u| \leq \delta$.
        By \cref{bon2}, we have
        \begin{align}
            \sup_{|u| \leq \delta} |g(w + u)| \leq c_2 (|g(w)| + \mu_1).
            \label{l4.4-01}
        \end{align}
		Furthermore, by \cref{bon4} with \(a = 2^{\tau +2 } d_0 \mu_2 c_2^{\tau }/(1 - \alpha)\), for $w + u > s$, there exists a constant $D_1$ depending on $c_2, c_3, d_0, d_1, \mu_1, \alpha$ and $\tau$ such that
        \begin{equ}
            \ML \sup_{|u| \leq \delta} |g(w + u) - g(w + u - s) | \\
            & \leq  \frac{1 - \alpha}{ 2^{\tau + 3} d_0 \mu_2 c_2^\tau  }\sup_{|u| \leq \delta} \lvert   g(w + u) \rvert  + D_1 (g(s) + 1).
            \label{l4.4-02}
        \end{equ}
        By \cref{l2-01,l4.4-01,l4.4-02}, we have
        \begin{align*}
            \ML \sup_{|u|\leq \delta} \bigl\lvert g(W + u)^{\tau - 1} (g(W + u) - g(W + u - s)) \bigr\rvert \\
            & \leq \Bigl(  \frac{1 - \alpha}{ 2^{\tau + 3} d_0 \mu_2 c_2^\tau  }\sup_{|u| \leq \delta} \lvert   g(W + u) \rvert  + D_1 (g(s) + 1)  \Bigr) \sup_{|u| \leq \delta} \lvert   g(W + u) \rvert^{\tau - 1}\\
        & \leq  \frac{1 - \alpha}{ 2^{\tau + 2} d_0 \mu_2 c_2^\tau  }\sup_{|u| \leq \delta} \lvert   g(W + u) \rvert^{\tau } +    \frac{ 2^{\tau + 3} d_0 \mu_2 c_2^\tau  }{\tau(1- \alpha)}     \times D_1^\tau (1 + g(s))^{\tau} \\
        & \leq \frac{1 - \alpha}{4 d_0 \mu_2} \Bigl(|g(W)|^\tau + \mu_1^\tau \Bigr) + \frac{ 2^{\tau + 3} d_0 \mu_2 c_2^\tau  }{\tau(1- \alpha)}     \times D_1^\tau (1 + g(s))^{\tau} \\
        & \leq  \frac{1 - \alpha}{4 d_0 \mu_2}  |g(W)|^\tau + D_2 (1 + g^{\tau}(s)),
     \end{align*}
     where
     \begin{align*}
         D_2 =  \frac{ 2^{2\tau + 3} d_0 \mu_2 c_2^\tau  }{\tau(1- \alpha)}     \times D_1^\tau +  \frac{(1 - \alpha) \mu_1^\tau}{4 d_0 \mu_2} .
     \end{align*}
     By \cref{d0,f:ratio},  we have
     \begin{equ}
         M_2 & \leq \frac{1 - \alpha}{4} \E |g(W)|^{\tau} (f(W ) + 1) \\
             & \qquad + d_0 \mu_2 D_2 (1 + g^{\tau} (s)) \E (f(W) + 1).
         \label{l4.4-04}
     \end{equ}
  For $M_3$, by \cref{lem:4.3} and similar to \cref{l4.4-04}, we have
     \begin{equ}
         M_3 & \leq \frac{ 1 - \alpha }{4} \E |g(W)|^{\tau} (f(W) + 1)  \\
         & \qquad + D_3 ( 1 + g^\tau (s) ) \E (f(W) + 1),
         \label{l4.4-05}
     \end{equ}
     where $D_3$ is a finite constant depending on $c_2, c_3, d_0, d_1, \mu_1, \alpha$ and $\tau$.

 For $M_4$, by \cref{d2,l2-01}, we have
     \begin{equ}
         M_4 & \leq \alpha \E |g(W)|^{\tau} f(W) + \alpha \E |g(W)|^{\tau - 1} f(W) \\
             & \leq \Bigl(\alpha + \frac{1 - \alpha}{4}\Bigr) \E |g(W)|^\tau f(W) + \Big(\frac{4 \alpha}{ 1 - \alpha }\Big)^{\tau -1} \E f(W).
         \label{l4.4-06}
     \end{equ}
By \cref{l4.4-0a,l4.4-00,l4.4-04,l4.4-05,l4.4-06}, we have
\begin{align*}
    \E |g(W)|^{\tau} f(W) & \leq \biggl( \alpha + \frac{3 (1 - \alpha) }{4}\biggr) \E |g(W)|^{\tau} f(W)   \\
                          & \qquad + (D_4 + \E|g(W)|^\tau)  (1 + g^\tau(s)) \E (f(W ) + 1),
\end{align*}
where $D_4$ is a constant depending on $c_2, c_3, d_0, d_1, \mu_1, \alpha$ and $\tau$. Rearranging the inequality gives
\begin{align}
    \label{l4.4-07}
    \E|g(W)|^{\tau} f(W) \leq \frac{4 (D_4 + \E|g(W)|^\tau)}{ 1 - \alpha }  (1 + g^\tau(s)) \E (f(W ) + 1).
\end{align}

Combining  \cref{l4.4-07,eq:re0}, we have
\begin{align}
    \label{l4.4-08}
    \E |g(W)|^{\tau} (f(W) + 1)
                         \leq D_5 (1 + g^\tau(s)) \E (f(W ) + 1),
\end{align}
where $D_5$ is a constant depending on $c_2, c_3, d_0, d_1, \mu_1, \alpha$ and $\tau$.
This proves \cref{eq:re1,eq:re2} for $\tau \geq 2$.

For $0\leq \tau < 2$ with $\E|g(W)|^2 <\infty.$ By the Cauchy inequality, we have
\begin{equation*} 
    (1 + g^{2-\tau}(s) )  |g(w)|^{\tau}\leq 1 +  g^2(s)+2 g^2(w),
\end{equation*}
and noting that for $s > 0$ and $g(s) > 0$,
\begin{equ}
    |g(w)|^\tau
	& \leq \frac{ 1 +  g^2(s)+2 g^2(w)}{1 + g^{2-\tau} (s)} \\
	& \leq g^{\tau }(s) + \frac{1 + 2g^2(w)}{ 1 + g^{2 - \tau} (s) }.
    \label{l4.4-10}
\end{equ}
By \cref{l4.4-08} with $\tau = 2$, we have
\begin{align}
    \label{l4.4-09}
    \E |g(W)|^{2} (f(W) + 1)
         \leq D_6 (1 + g^2(s)) \E (f(W ) + 1),
\end{align}
where $D_6$ is a constant depending on $c_2, c_3, d_0, d_1, \mu_1, \alpha$ and $\tau$.

Thus, for $0 \leq \tau < 2$, by \cref{l4.4-09,l4.4-10}, we have
\begin{align*}
    \E | g(W)|^{\tau} ( f(W) + 1 ) & \leq g^{\tau } (s) \E (f(W) + 1) \\
                                   & \quad + \frac{\E(f(W) + 1)  + 2 \E g^2(W) (f(W) + 1)}{ 1 + g^{2 - \tau}(s) } \\
                                   & \leq D_7 (1 + g^\tau(s) ) \E (f(W) + 1),
\end{align*}
where $D_7$ is a constant depending on $c_2, c_3, d_0, d_1, \mu_1, \alpha$ and $\tau$.
This completes the proof together with \cref{l4.4-08}. 
\end{proof}

\begin{lemma}\label{lem:4.5}
Let \(f(w, s)\) be defined as in \cref{Fws}.
Let $ 0 < \delta \leq 1$ and $s>0$.
Suppose that the conditions in \cref{thm:2.1} are satisfied.
Then, we have
\begin{eqnarray}
\lefteqn{
    \E (f(W, s) + 1) } \nonumber \\
     & \leq & C (1 + s) \exp \Bigl\{ C \Bigl(\delta \bigl(1 + s g^2(s) \bigr)+  \delta_1 \bigl(1 + s g^{\tau_1 + 1} (s)\bigr)   \\
     && \quad \quad \quad \quad \quad \quad \quad  \quad \  +  \delta_2 \bigl(1 + s g^{\tau_2}(s)\bigr) \Bigr) \Bigr\}. \nonumber
     \label{l4.5a}
\end{eqnarray}
\end{lemma}

\begin{remark}
\label{rem:a4.5}
Following the proof of \cref{lem:4.5}, if we assume that the condition \cref{de1} is replaced by \cref{r2.1-b,r2.1-c}, then the result of \cref{lem:4.5} still holds.
\end{remark}
\begin{proof}[Proof of \cref{lem:4.5}]
    Let $h(s)=\E f(W,s)$ and let $f(w) := f(w , s)$. By \cref{eq:fpartialw,eq:fpartials}, for $s>0$, we have
\begin{align*}
h'(s)={}&\E\left(e^{G(W)-G(W-s)}g(W-s)\I(W>s)\right)\nn\\
={}& \E(f(W)g(W))+E(g(W)\I(W>0))-E(f'(W)). \nn
\end{align*}

 We first show that $h'(s)$ can be bounded by
a function of $h(s)$. We then solve the differential inequality to obtain
the bound of $h(s)$, using an idea similar to that in the proof of \cref{lem:4.4}.

By \eqref{c1}, we have
\begin{equ}
    {\E(f(W)g(W))} - \E (f'(W))
={} & T_1 + T_2 + T_3,
\label{l4.5-01}
\end{equ}
where
\begin{align*}
    T_1 & = \E\Bigl(\intud \bigl(f'(W+u) - f'(W) \bigr)\hatK (u)du\Bigr), \\
    T_2 & = \E f'(W) ( \E(\hatK_1 |W) - 1) , \\
    T_3 & = \E(f(W)R(W)). \nn
\end{align*}
We next give the bounds of $T_1, T_2$ and $T_3$.
\begin{enumerate}[label*=\roman*).]
    \item The bound of $T_1$.
        By \cref{f''}, we have
        \begin{align*}
            & \sup_{|u|\leq \delta} |f'(w + u) - f'(w)| \\
            & \leq \delta \sup_{|u|\leq \delta} |f''(w + u) | \\
            & \leq \delta \mu_3 (g^2(w) + 1) (f(w) + 1).
        \end{align*}
        By \cref{d0} and \cref{lem:4.4}, we have
        \begin{equ}
            |T_1| & \leq \delta d_0 \mu_3 \E ( g^2(W) + 1 ) (f(W) + 1) \\
                  & \leq D_8 \delta (1 + g^2(s)) \E (f(W) + 1),
            \label{l4.5-02}
        \end{equ}
        where $D_8$ is a constant depending on  $c_2, c_3, d_0, d_1, \mu_1$ and  $\alpha$.
    \item The bound of $T_2$.
        By \cref{de1,lem:4.4}, we have
        \begin{equ}
            |T_2 | & \leq \delta_1 \E \Bigl( |g(W)| \bigl( |g(W)|^{\tau_1} + 1\bigr) (f(W) + 1) \Bigr) \\
                   & \leq 2 \delta_1 \E \bigl( |g(W)|^{\tau_1 + 1} + 1\bigr) \bigl( f(W) + 1\bigr) \\
                   & \leq D_9 \delta_1 (1 + g^{\tau_1 + 1} (s) ) \E \bigl( f(W) + 1\bigr),
            \label{l4.5-03}
        \end{equ}
        where $D_9$ is a constant depending on  $c_2, c_3, d_0, d_1, \mu_1, \tau_1$ and  $\alpha$.
    \item The bound of $T_3$. By \cref{de2,lem:4.4}, we have
        \begin{equ}
            \label{l4.5-04}
            T_3 & \leq \delta_2  \E ( |g(W)|^{\tau_2}  + 1) f(W) \\ 
                & \leq D_{10} \delta_2 (1 + g^{\tau_2}(s)) \E \bigl( f(W) + 1\bigr),
        \end{equ}
        where $D_{10}$ is a constant depending on  $c_2, c_3, d_0, d_1, \mu_1, \tau_2$ and  $\alpha$.
\end{enumerate}
By \cref{eq:re0}, we have
\begin{align}
    \E g(W) \I(W > 0) \leq D_{11},
    \label{l4.5-05}
\end{align}
where $D_{11}$ is a constant depending on  $c_2, c_3, d_0, d_1, \mu_1$ and  $\alpha$.
By \cref{l4.5-01,l4.5-02,l4.5-03,l4.5-04,l4.5-05}, we have
\begin{align*}
    h'(s) & \leq D_{11} + D_{12} \bigl(\delta \bigl(1 + g^2(s) \bigr) +  \delta_1 \bigl(1 + g^{\tau_1 + 1} (s)\bigr) +  \delta_2 \bigl(1 + g^{\tau_2}(s)\bigr) \bigr) \\
          & \hspace{2cm} \times \E \bigl(f(W) + 1\bigr) ,
\end{align*}
where $D_{12} = \max \bigl( D_{8}, D_{9}, D_{10}\bigr)$. Therefore,
\begin{align*}
    h'(s) & \leq D_{12} \bigl(\delta \bigl(1 + g^2(s) \bigr) +  \delta_1 \bigl(1 + g^{\tau_1 + 1} (s)\bigr) +  \delta_2 \bigl(1 + g^{\tau_2}(s)\bigr) \bigr) h(s) \\
          & \quad + D_{11} + D_{12} \bigl(\delta \bigl(1 + g^2(s) \bigr) +  \delta_1 \bigl(1 + g^{\tau_1 + 1} (s)\bigr) +  \delta_2 \bigl(1 + g^{\tau_2}(s)\bigr) \bigr),
\end{align*}
By solving the differential inequality and given that
$s + sg^\tau (s) \leq 1+ ( 1+ g^{-\tau}(1) ) s g^\tau(s)$ for $ \tau>0$ and $s\geq 0$,  we have
\begin{align*}
    \E (f(W) + 1) & \leq C_1 (1 + s) \exp \Bigl\{ C_2 \Bigl(\delta \bigl(1 + s g^2(s) \bigr) +  \delta_1 \bigl(1 + s g^{\tau_1 + 1} (s)\bigr) \\
                  & \hspace{4cm} +  \delta_2 \bigl(1 + s g^{\tau_2}(s)\bigr) \Bigr) \Bigr\},
\end{align*}
where $C_1$ and $C_2$ are constants depending on  $c_2, c_3, d_0, d_1, \mu_1, \tau_1, \tau_2$ and  $\alpha$.
This completes the proof.
\end{proof}

The next lemma gives the properties of the Stein solution.

\begin{lemma}\label{lem:4.6}
Let $f_z$ be the solution to Stein's equation {\eqref{ste:eq}}. 
Then, for $z\geq 0$,
\begin{gather}\label{fg}
|f_z(w)g(w)|\leq \begin{cases}
1-F(z),& w\leq 0,\\
F(z),& w>0,
\end{cases} \\
\label{fnorm}
0 \leq f_z(w)\leq \begin{cases}
(1-F(z))/c_1, & w\leq 0,\\
F(z)/c_1, & w>0,
\end{cases}
\intertext{and}
\label{f'norm}
|f_z'(w)|\leq \begin{cases}
2(1-F(z)),& w\leq 0,\\
1,& 0<w\leq z,\\
2F(z),& w>z.
\end{cases}
\end{gather}
\end{lemma}
\begin{proof}[Proof of \cref{lem:4.6}]

Our first step is to prove \eqref{fg}. By \eqref{eq:solution}, we have
\begin{equation}
f_z(w)g(w)=\begin{dcases}
    \frac{F(w)g(w)(1-F(z))}{ p(w)},& w\leq z,\\[5mm]
    \frac{F(z)g(w)(1-F(w))}{ p(w)},& w>z.
\end{dcases}
\end{equation}
Without loss of generality, we must consider only three case when $z>0$:
\begin{enumerate}
\item $w < 0$: By \eqref{wg<0},
\[
|f_z(w)g(w)|\leq 1-F(z).
\]
\item $0\leq w\leq z$: Since $w\leq z,$ $1-F(z)\leq 1-F(w),$ thus by \eqref{wg0},
\[
|f_z(w)g(w)|\leq \frac{F(w)|g(w)|(1-F(w))}{ p(w)}\leq F(w)\leq F(z).
\]
\item $w>z$: By \eqref{wg0},
\[
|f_z(w)g(w)|\leq F(z).
\]
\end{enumerate}
We can have a similar argument when $z\leq 0,$ which completes the proof of \eqref{fg}.
Additionally, \eqref{fnorm} can be shown similarly.
\eqref{f'norm} follows directly from \eqref{ste:eq} and \eqref{fg}.
\end{proof}

\begin{lemma} \label{lem:4.7}
For $z>0$ and  $0 \leq \tau \leq \max(2, \tau_1+1, \tau_2)$,  
\begin{equation}\label{eq:bod}
\E(f_z(W)|g(W)|^{\tau})\leq C(1+zg^{\tau}(z))(1-F(z)),
\end{equation}
provided that $\max(\delta , \delta_1, \delta_2) \leq 1$ and $\delta zg^2(z)+\delta_1 zg^{\tau_1+1}(z)+ \delta_2 zg^{\tau_2}(z) \leq 1.$
\end{lemma}

\begin{proof}[Proof of \cref{lem:4.7}]
By \eq{eq:solution},
\begin{align*}
{\E(f_z(W)|g(W)|^{\tau}})
= {} & T_4+T_5+T_6,
\end{align*}
where
\begin{align*}
	T_4 & =  F(z) \E\lme  \frac{1-F(W)}{ p(W)}|g(W)|^\tau \I(W>z) \rme, \\
	T_5 & = (1-F(z))\E\lme  \frac{F(W)}{ p(W)}|g(W)|^\tau \I(W<0) \rme, \\
	T_6 & = (1-F(z))\E\lme  \frac{F(W)}{ p(W)}|g(W)|^\tau \I(0\leq W\leq z) \rme.
\end{align*}
\begin{enumerate}[label*=\roman*{).}]
    \item For $T_4$,  we first consider the case when $\tau \geq 1$.
As $g(w)$ is increasing, $e^{G(w)-G(w-z)}$ is also increasing with respect to $w$; thus,
\[
\I(W>z)\leq \frac{e^{G(W)-G(W-z)}\I(W>z)}{ e^{G(z)}}.
\]
By \cref{lem:4.5}, we have $\max (\delta, \delta_1, \delta_2) \leq 1$ and $z$, satisfying that $\delta zg^2(z)+\delta_1 zg^{\tau_1+1}(z)+ \delta_2 zg^{\tau_2}(z) \leq 1,$
\begin{align*}
    \E (f(W, z) + 1) \leq C (1 + z).
\end{align*}
Hence, by
\cref{wg0,lem:4.4}, we have
 \begin{equ}
     T_4  \leq {} &  C e^{-G(z)} \E |g(W)|^{\tau - 1} e^{G(W)-G(W-z)} \I(W > z)   \nn\\
     \leq {}& C e^{-G(z)} (1 +  g^{\tau - 1}(z) ) \E (f (W, z) + 1)\\
     \leq {}& C e^{-G(z)} (1 + z g^{\tau - 1}(z) )  \\
     \leq {}& C(1+ zg^\tau (z)) (1-F(z)),\lbl{T1}
 \end{equ}
for $\max (\delta, \delta_1, \delta_2) \leq 1$ and $z$, satisfying that $\delta zg^2(z)+\delta_1 zg^{\tau_1+1}(z)+ \delta_2 zg^{\tau_2}(z) \leq 1.$
If $0 \leq \tau < 1$, then $g^{\tau} (w) \leq 2 \bigl(1 + g(w)\bigr)/\bigl( 1 + g^{1 - \tau}(z)\bigr)$ for $w > z$.
Therefore, \cref{T1} also holds for $0 \leq \tau <  1.$
\item As to $T_5$, because $F(w)/p(w)\leq 1/c_1$ for $w\leq 0$,
\[
    T_5\leq \frac{1}{c_1}(1-F(z))\E|g(W)|^{\tau} \I(W < 0).
\]
By  \cref{eq:re0}, we have
\begin{equation}\label{T2}
T_5\leq C(1-F(z))
\end{equation}
for some constant $C$.
\item
We now bound $T_6$. By \cref{lem:4.5,lem:4.4},
\begin{equ}\lbl{T3}
T_6\leq {}& C(1-F(z))\E e^{G(W)}|g(W)|^{\tau}\I(0\leq W\leq z)\nn\\
\leq {}& C(1-F(z))(1+g^{\tau}(z))\E e^{G(W)}\I(0\leq W\leq z)\nn\\
\leq {}& C(1-F(z))(1+zg^{\tau}(z)).
\end{equ}
\end{enumerate}
By \cref{T1,T2,T3}, we have
\[
\E(f_z(W)|g(W)|^{\tau})\leq C(1+zg^{\tau}(z))(1-F(z)),
\]
which completes the proof.
\end{proof}

\subsection{Proofs of Propositions 4.1 and 4.2} 
\label{sub:proposition}

We are now ready to give the proofs of \cref{pro:4.1,pro:4.2}.

\begin{proof}[Proof of \cref{pro:4.1}]
    Recalling  \cref{d0}, we  have
    \begin{equ}
        I_1 & \leq d_0 \E \Bigl( \sup_{|t|\leq \delta} \bigl\lvert f_z(W + t) g(W + t) - f_z(W) g(W)\bigr\rvert\Bigr)\\
            & \leq \delta d_0 \E \sup_{|t| \leq \delta}\left|  (f_z(W+ t )g(W+ t ))'
  \right|.
  \label{i1-00}
    \end{equ}
We first prove \eqref{IA}. By \cref{lem:4.6}, $\|f_z\|\leq 1/c_1$
and $\|f_z'\|\leq 2$. Thus, for $0 < \delta \leq 1$,
\begin{equ} \label{4.31}
\ML {\E\Bigl(\sup_{| t | \leq \delta}|\bigl(f_z(W+ t )g(W+ t )\bigr)'|\Bigr)}\nn\\
 \leq {} &  \E \Bigl(\sup_{| t |\leq \delta} \bigl( \left|  f_z(W+ t ) g'(W+ t )  \right|
+ \left|  f_z'(W+ t ) g(W+ t )  \right|\bigr)\Bigr)\nn\\
 \leq {} &  (2+1/c_1) \E\Bigl(\sup_{| t |\leq \delta} \bigl( |g'(W+ t )| + |g(W+ t )|
\bigr) \Bigr) \nn\\
\leq {} &  4 c_3 (1 + 1/c_1) (1 + c_2)  \bigl( \E|g(W)| + \mu_1 \bigr), 
\end{equ}
where in the last inequality we use \eqref{c3} and \cref{lem:4.2}.
This proves \eqref{IA} by  \cref{4.31,i1-00,eq:re0}.

Next, we prove \eqref{IB}. Similar to the proof of \cref{IA}, we first calculate the
following term: $$\E\Bigl(\sup_{| t | \leq \delta}|(f_z(W+ t )g(W+ t ))'|\Bigr).$$
Note that
\begin{equ}
  \lbl{fg'}
  & (f_z(w)g(w))' \\
  & =  \begin{cases}
    \frac{p(w)g(w) + F(w)g'(w) + F(w)g^2(w) }{ p(w)} (1-F(z)), & w\leq z,\\
    \frac{-p(w)g(w) + (1-F(w))g'(w) + (1-F(w))g^2(w) }{ p(w)} F(z), & w>z.
  \end{cases}
\end{equ}
For $w+ t \leq 0$, by \cref{wg<0}, we have
\begin{align*}
    \ML |(f_z(w+ t )g(w+ t ))'| \\
    & \leq (1-F(z))\Big( 2|g(w+ t )| + \frac{g'(w+ t ) }{\max\{ c_1,
|g(w+ t )|\})} \Big)\nn\\& \leq (1-F(z)) \left( 2\left|  g(w+ t )  \right| +
c_3(1+1/c_1) \right) \\
& \leq C (1 - F(z)) ( |g(w)| + 1).
\end{align*}
Thus, by \cref{eq:re0},
\begin{equation}
\E\biggl(\sup_{\left|   t   \right|\leq \delta}|(f_z(W+ t )g(W+ t ))'|\I(W+ t \leq 0) \biggr) \leq C(1-F(z)).
\label{i1-01}
\end{equation}

For $w+ t >z$, and $|t| \leq \delta$, again by \cref{lem:4.2}, we have
\begin{equ}
    \ML |(f_z(w+ t )g(w+ t ))'| \\
     \leq {} &  F(z) \Bigl( |g(w+ t )| + \frac{1-F( w + t )}{ p(w + t)} (|g'(w+ t )| + |g(w+ t )|^2) \Bigr)\nn\\
 \leq {} &  C  (1+\left|  g(w+ t )
\right| )  \nn\\
\leq {} & C  (|g(w) | + 1). \nn
\end{equ}
Hence, by \cref{lem:4.4,lem:4.5}, we have
\begin{equ}
    \ML \E \sup_{|t| \leq \delta }|(f_z(W+ t )g(W+ t ))'| \I(W + t \geq z)  \\
    \leq {} &  C\E\Bigl(\bigl(|g(W)| + 1\bigr)\I(W>z-\delta)\Bigr)\nn\\
 \leq {} &  C p(z-\delta) \E \left(e^{G(W)-G(W-z+\delta)} |g(W)| \I(W>z-\delta)\right)\nn\\
 \leq {} &  C e^{\delta g(z)} p(z) (1+g(z)) \E\left(e^{G(W)-G(W-z+\delta)}\I(W>z-\delta)\right)\nn\\
 \leq {} &  C e^{\delta g(z)} (1+zg^2(z)) ( 1 - F(z) ), \lbl{eq:wgz}
\end{equ}
where we use the \cref{lem:4.1} in the last line.
Also note that by \cref{eq:delta-gz},  $\delta g(z) \leq \delta + \delta z g^2(z) \leq 2$ for $z\geq 1$ and $\delta g(z) \leq \mu_1 $ for $ 0 \leq z \leq 1$. Hence,
\begin{equ}
	\label{eq:deltagz-bound}
	\delta g(z) \leq \max (2, \mu_1).
\end{equ}
Thus, \cref{eq:wgz,eq:deltagz-bound} yield
\begin{equ}
& \E\biggl(\sup_{\left|   t   \right|\leq \delta} |(f_z(W+ t )g(W+ t ))'|\I(W+ t >z)\biggr) \\
& \leq C(1+zg^2(z))(1-F(z)).
\label{i1-02}
\end{equ}

For $w+ t \in (0,z)$ and $|t| \leq \delta$,  by \cref{bon2,fg',eq:deltagz-bound}, we have
\begin{equ}
\ML |(f_z(w+ t )g(w+ t ))'|\\
& \leq C (1 - F(z)) e^{G(w + t)} \bigl(1 + g(w + t)^2\bigr) \\
& \leq C (1 - F(z)) e^{G(w) + \delta g(z) } \bigl(1 + |g(w)|^2\bigr) \\
& \leq C (1 - F(z) ) e^{G(w)} \bigl( 1 + |g(w)|^2 \bigr). \nn
\end{equ}
By \cref{lem:4.4,lem:4.5,bon2}, we have
\begin{equ}
\ML \E\biggl(\sup_{\left|   t   \right|\leq \delta} |(f_z(W+ t )g(W+ t ))'|\I(0\leq W+ t \leq z) \biggr) \\
& \leq C (1 - F(z)) \E e^{G(W)} ( 1 + |g(W)|^2 ) \I( -\delta \leq W \leq z + \delta ) \\
& = C (1 - F(z)) \E e^{G(W)} ( 1 + |g(W)|^2 ) \I( -\delta \leq W \leq 0) \\
& \qquad  +C (1 - F(z) )  \E e^{G(W)} ( 1 + |g(W)|^2 ) \I( 0 \leq W \leq z + \delta )
 \\
& \leq C e^{\mu_1 } (1 + \mu_1^2) (1 - F(z)) \\
& \qquad  + C (1 - F(z)) \bigl( 1 + (z+\delta) g^2(z + \delta)\bigr) \\
 & \leq C (1 - F(z)) (1 + zg^2(z)).
 \label{i1-03}
\end{equ}
Putting together \cref{i1-01,i1-02,i1-03} gives
\begin{equation}
  \E\biggl(\sup_{\left|   t   \right|\leq \delta} |(f_z(W+ t )g(W+ t ))'|\biggr)\leq
  C(1+zg^2(z)) (1-F(z)).
  \label{eq:fg'bound}
\end{equation}
By \cref{i1-00,eq:fg'bound}, we obtain \cref{IB}.
\end{proof}

\begin{proof}
    [Proof of \cref{pro:4.2}]

    By \cref{lem:4.6}, we have $\|f_zg\|\leq 1$; thus, by \cref{de1,eq:re0},
    \begin{align*}
        I_2+I_3&\leq C \E|\E(\hatK_1 |W)-1|\leq C \delta_1\bigl( \E(|g(W)|^{\tau_1}) + 1\bigr) \leq C\delta_1 .
    \end{align*}

    To bound $I_4$, by \cref{de2,eq:re0,fnorm}, we have
    \begin{align*}
    I_4\leq C \delta_2.
    \end{align*}
    This proves \cref{pro:4.2-00}.

    We now move to prove \cref{pro:4.2-01,pro:4.2-02}. As to $I_2$,
    by \cref{de1} and  \cref{lem:4.7}, for $z \geq 0, \max (\delta, \delta_1, \delta_2) \leq 1$ and $\delta z g^2(z) + \delta_1 z g^{\tau_1 + 1} (z) + \delta_2 z g^{\tau_2} (z) \leq 1$, we have
\begin{equ}
I_2&\leq \delta_1 \E\left(f_z(W)|g(W)|(|g(W)|^{\tau_1}+1\right)\nn\\
&\leq  C \delta_1 \E\left(f_z(W)(1+|g(W)|^{\tau_1+1})\right) \nn\\
&\leq  C\delta_1(1+zg^{\tau_1+1}(z)) (1 - F(z) ) .\label{I2}
\end{equ}
As to $I_3$, note that
\[
\I(W>z)\leq \frac{e^{G(W)-G(W-z)}}{ e^{G(z)}}\I(W>z).
\]
By \cref{lem:4.4,lem:4.5},
\begin{equ}
\ML \E\bigl((1+|g(W)|^{\tau_1})\I(W>z)\bigr)\\
\leq {}& Cp(z)\E\left(e^{G(W)-G(W-z)}(1+|g(W)|^{\tau_1})\I(W>z)\right)\\
\leq {}& C(1+zg^{\tau_1}(z))p(z) \\ 
\leq {}& C (1 + zg^{\tau_1 + 1} (z) ) ( 1 - F(z) ),
\label{eq:477}
\end{equ}
where we use \cref{wg0} in the last inequality. Thus, by \cref{lem:4.4} and \cref{eq:477},
\begin{equ}
I_3\leq {}&  \delta_1 \bigl(1-F(z)\bigr)E\bigl( |g(W)|^{\tau_1} + 1 \bigr) \\
          & + \delta_1 \E\Bigl( \bigl( |g(W)|^{\tau_1}+ 1\bigr)\I(W>z + \delta) \Bigr)\nn\\
\leq {} & \delta_1 \bigl(1-F(z)\bigr)E\bigl( |g(W)|^{\tau_1} + 1 \bigr) \\
        & + \delta_1 \E\Bigl( \bigl( |g(W)|^{\tau_1}+ 1\bigr)\I(W>z) \Bigr)\nn\\
\leq {}& C\delta_1(1+zg^{\tau_1+1}(z)) ( 1 - F(z) ) . 
\lbl{I3}
\end{equ}
\cref{pro:4.2-01} now follows  by  \cref{I2,I3}.

As to $I_4$, because $|R(W)|\leq \delta_2(1+|g(W)|^{\tau_2}),$ by \cref{eq:bod}, we have
\begin{align}
I_4\leq C\delta_2(1+zg^{\tau_2}(z))(1-F(z)) . 
\lbl{I4}
\end{align}
This completes the proof of  \cref{pro:4.2}.
\end{proof}


\subsection{Proof of Remark 2.1} 
In this subsection, we assume that the condition \eqref{de1} in \cref{thm:2.1}  is replaced by
\crefrange{r2.1-a}{r2.1-c}, then the result of \cref{r2.1} follows from the proof of \cref{thm:2.1}, \cref{pro:4.1,pro:4.2} and the following proposition: 

\begin{proposition}
    Assume that the condition \eqref{de1} in \cref{thm:2.1}  is replaced by
    \crefrange{r2.1-a}{r2.1-c}, then \eqref{pro:4.2-00} and \eqref{pro:4.2-01} hold.
    \label{pro:4.3}
\end{proposition}

\begin{proof}
    [Proof of \cref{pro:4.3}]
    Following the proof of \cref{pro:4.2}, it suffices to prove the following inequalities:
    \begin{gather}
        \E | K_2 | \leq \delta_1, \label{p4.3-01}
        \intertext{and for $z>0$ such that $\delta z g^2(z) + \delta_1 z g^{\tau_1 + 1} (z) + \delta_2 z g^{\tau_2} (z)  \leq 1$,}
        \E |f_z(W) g(W) K_2 | \leq C \delta_1 ( 1 + z g^{\tau_1 + 1} (z)) ( 1 - F(z)), \label{p4.3-02}\\
        \E |K_2| I(W > z) \leq C \delta_1 ( 1 + z g^{\tau_1 + 1} (z)) ( 1 - F(z)). \label{p4.3-03}
    \end{gather}

    For \cref{p4.3-01}, by \cref{r2.1-c} with $s = 0$, noting that  $\zeta(W, 0) \equiv 1$ and $g(0) = 0$, we have \cref{p4.3-01} holds.

    For \cref{p4.3-02}, by the definition of $f_z$, and by \cref{lem:4.1,lem:4.6},  we have
    \begin{equ}
        \label{p4.3-04}
        \E | f_z(W) g(W) K_2 | & \leq
                               T_7 + T_8 + T_9,
    \end{equ}
    where
    \begin{align*}
        T_7 & = (1 - F(z)) \E |K_2| I(W < 0), \\
        T_8 & =  (1 - F(z)) \E |K_2| g(W) e^{G(W)} I(0 \leq W \leq z) , \\
        T_9 & = \E |K_2| I(W > z).
    \end{align*}

    For $T_7$, by \cref{p4.3-01}, we have
    \begin{equ}
        \label{p4.3-05}
        T_7 \leq \delta_1 (1 - F(z)).
    \end{equ}
    For $T_8$, by the monotonicity of $g(\cdot)$ and by \cref{r2.1-c} and \cref{rem:a4.5}, we have
    \begin{equ}
        \label{p4.3-06}
        T_8 & \leq (1 - F(z) ) g(z) \E |K_2| \zeta(W, z)\\
            & \leq \delta_1 (1 - F(z)) ( 1 + g(z)^{\tau_1 + 1}) \E \zeta(W, s) \\
            & \leq C \delta_1 ( 1 + z g^{\tau_1 + 1}(z)) (1 - F(z)).
    \end{equ}
    For $T_9$, by the Chebyshev inequality, by \cref{r2.1-c,lem:4.5,lem:4.1}, we have
    \begin{equ}
        \label{p4.3-07}
        T_9 & \leq e^{-G(z)} \E |K_2| \zeta(W, z) I(W > z) \\
            & \leq C \delta_1 (1 + z g^{\tau_1}(z)) e^{-G(z)} \\
            & \leq  C \delta_1 ( 1 + z g^{\tau_1 + 1}(z)) (1 - F(z)).
    \end{equ}

    The inequality \cref{p4.3-02} follows from  \crefrange{p4.3-04}{p4.3-07}  while \cref{p4.3-03} follows from \cref{p4.3-07}. This completes the proof. \qedhere
\end{proof}

\subsection{Proof of Remark 2.2} 
In this subsection, we assume that the condition \cref{d2} is replaced by \cref{r2.2-a,r2.2-b}.
The conclusion of \cref{r2.2} follows from the proof of \cref{thm:2.1} and the following lemma.

\begin{lemma}
    \label{lem:4.10}
    Let the conditions in \cref{r2.2} be satisfied. Furthermore, $0 < \delta \leq 1$, and $s > 0$ such that
    $ \delta s g^2 (s) \leq 1$. For $0 \leq \tau \leq \max \left\{ 2 , \tau_1 + 1, \tau_2 \right\}$,
    inequalities \cref{eq:re0,eq:re1,eq:re2} hold.
\end{lemma}

\begin{proof}
    Recall that
    $s_0 = \max \left\{ s: \delta s g^2 (s) \leq 1 \right\}$ and $\delta \leq 1$. We have
    \begin{align*}
        s_0 \geq s_1  \quad\text{and} \quad  \delta s_1 g^2(s_1) = 1.
    \end{align*}
    Following the proof of \cref{lem:4.4}, it suffices to prove the following two inequalities.

    \noindent
    For $Q_2$ defined in \cref{l2-0g},
    \begin{align}
        Q_2 \leq \biggl( \alpha + \frac{ 1 - \alpha }{ 4 }\biggr) \E g_{+}^{\tau} (W) + C,
        \label{r2.2-01}
    \end{align}
    and for $M_4$ defined in \cref{l4.4-0m4},
    \begin{align}
        \label{r2.2-02}
        M_4
        & \leq \biggl(\alpha + \frac{1 - \alpha}{4}\biggr) \E |g(W)|^\tau f(W) + \biggl(\frac{4 \alpha}{ 1 - \alpha }\biggr)^{\tau -1} \E f(W) + C.
    \end{align}
    For $Q_2$, by \cref{r2.2-a} and similar to \cref{l2-q2}, we have
    \begin{align*}
        Q_2 & \leq \biggl( \alpha + \frac{ 1 - \alpha }{ 4 }\biggr) \E g_{+}^{\tau} (W) + \biggl( \frac{ 4 }{ 1 - \alpha}\biggr)^{ \tau - 1} + d_2 \E g_{+}^{\tau} (W) \I(W > \kappa).
    \end{align*}
    For the last term, by \cref{d1,r2.2-b} and noting that $0 \leq \tau \leq  \max \left\{ 2 , \tau_1+ 1, \tau_2 \right\}$, we obtain
    \begin{equ}
        d_2 \E g_{+}^{\tau} (W) \I (W > \kappa) & \leq d_1^{-\tau}d_2 \delta^{-\tau} \P (W > \kappa) \\
                                                & \leq d_1^{-\tau}d_3 \delta^{-\tau} \exp ( - 2 s_0 d_1^{-1}\delta^{-1}) \\
                                                & \leq d_1^{-\tau} d_3 \sup_{\delta > 0} \{\delta^{-\tau} \exp ( - 2 s_1 d_1^{-1} \delta^{-1})\} \\
                                                & =  d_3 \Bigl( \frac{\tau}{ 2 s_1}\Bigr)^{\tau} e^{-\tau},
        \label{r2.2-03}
    \end{equ}
    where the equality holds when $\delta = {2 s_1}/{(d_1 \tau)  }$.
    The inequality \cref{r2.2-01} follows from \cref{l2-q2,r2.2-03}.

    As to $M_4$, by \cref{r2.2-a}, we have
    \begin{align*}
        M_4
        & \leq \biggl(\alpha + \frac{1 - \alpha}{4}\biggr) \E |g(W)|^\tau f(W) + \biggl(\frac{4 \alpha}{ 1 - \alpha }\biggr)^{\tau -1} \E f(W) \\
        & \quad\quad + d_2 \E \bigl\lvert g^{\tau} (W)\bigr\rvert e^{ G(W) - G(W - s)} \I ( W > \kappa).
    \end{align*}
    For the last term, by \cref{d1,r2.2-b} and noting that $g(\cdot)$ is nondecreasing and $s \leq s_0$, similar to \cref{r2.2-03}, we have
    \begin{equ}
        \label{r2.2-04}
        \ML  d_2 \E \bigl\lvert g^{\tau} (W)\bigr\rvert e^{ G(W) - G(W - s)} \I ( W > \kappa) \\
        & \leq d_1^{-\tau}d_2 \delta^{-\tau} e^{ s d_1^{-1} \delta^{-1} } \P (W > \kappa) \\
        & \leq d_1^{-\tau} d_3 \delta^{-\tau} e^{ - s_0 d_1^{-1} \delta^{-1}} \\
        & \leq d_1^{-\tau} d_3 \sup_{\delta > 0} \{\delta^{-\tau} e^{-s_1 d_1^{-1} \delta^{-1}}\} \\
        & =  d_3 \Bigl(\frac{ \tau}{ s_1}\Bigr)^{\tau} e^{-\tau},
    \end{equ}
    where the equality holds when $\delta = {s_1 }/{ (d_1\tau)}$. Combining \cref{l4.4-06,r2.2-04}, inequality \cref{r2.2-02} holds.
    Following the proof of \cref{lem:4.4} and replacing \cref{l2-q2} and \cref{l4.4-06} with \cref{r2.2-01} and \cref{r2.2-02}, respectively, we complete the proof of \cref{lem:4.10}.
\end{proof}



\section{Proofs of Theorems 3.1--3.2}\label{proof2}

\setcounter{equation}{0}

\subsection{Proof of Theorem 3.1}
\def\F{\mathcal{F}}
In this subsection, we use \cref{r2.1,r2.2} to prove the result.

We first prove some preliminary lemmas.
\begin{lemma}
    Let $\xi \sim \rho$. For $s \in \R$, define
    \begin{align*}
        \psi_n (s) = \frac{ \E \bigl( \xi e^{ \frac{\xi^2}{2n} +  \xi s }  \bigr) }{ \E \bigl( e^{ \frac{\xi^2}{2n}  +  \xi s} \bigr) }, \quad
        \psi_{\infty} (s) = \frac{ \E \bigl( \xi e^{\xi s} \bigr) }{ \E \bigl( e^{\xi s} \bigr) },
        \intertext{and}
        \phi_n (s) = \frac{ \E \bigl( \xi^2 e^{ \frac{\xi^2}{2n} +  \xi s }  \bigr) }{ \E \bigl( e^{ \frac{\xi^2}{2n}  +  \xi s} \bigr) }, \quad
        \phi_{\infty} (s) = \frac{ \E \bigl( \xi^2 e^{\xi s} \bigr) }{ \E \bigl( e^{\xi s} \bigr) }.
    \end{align*}
    Let $m = \frac{1}{n} \sum_{i = 1}^n X_i$ and $m_i = \frac{1}{n} \sum_{j \neq i} X_j$.
    We have
    for each $1 \leq i \leq n$,
    \begin{align}
        \bigl\lvert \psi_{\infty} (m) - \psi_n (m_i) \bigr\rvert \leq C n^{-1},
        \label{l5.1-01}\\
        \bigl\lvert \phi_{\infty} (m) - \phi_n (m_i) \bigr\rvert \leq C n^{-1},
        \label{l5.1-02}
    \end{align}
    where $C$ is a positive constant depending only on $L$.
    \label{lem:5.1}
\end{lemma}
\begin{proof}
    [Proof of \cref{lem:5.1}]
    Recall that $|\xi| \leq L$ and observe that
    \begin{align*}
        \Bigl\vert \E \Bigl( \xi \bigl( e^{\frac{\xi^2}{2n} + \xi s} - e^{\xi s} \bigr) \Bigr) \Bigr\vert  & \leq
        \frac{1}{2n} \E |\xi|^3 e^{\frac{\xi^2}{2n} + \xi s} \leq \frac{L^3}{2 n} e^{L^2/2} \E e^{\xi s}, \\
        \Bigl\vert \E \Bigl( e^{\frac{\xi^2}{2n} + \xi s} - e^{\xi s}\Bigr) \Bigr\vert  & \leq
        \frac{1}{2n} \E |\xi|^2 e^{\frac{\xi^2}{2n} + \xi s} \leq \frac{L^2}{2 n} e^{L^2/2} \E e^{\xi s}, \\
        \bigl\lvert \E \xi e^{\xi s} \bigr\rvert & \leq L \E e^{\xi s},
    \end{align*}
    and
    \begin{align*}
        \E \bigl( e^{\frac{\xi^2}{2n} + \xi s}  \bigr) \geq \E e^{\xi s}.
    \end{align*}
    Hence,
    \begin{equ}
        \label{l5.1-03}
        |\psi_n (s) - \psi_{\infty} (s) | & \leq \frac{  \bigl\lvert \E e^{\xi s} \bigr\rvert \times \Bigl\lvert \E \xi e^{\frac{\xi^2}{2n} + \xi s} - \E \xi e^{\xi s}\Bigr\rvert   }{ \E e^{\frac{\xi^2}{2n} + \xi s } \E e^{\xi s} } \\
                                          & \quad + \frac{  \bigl\lvert \E \xi e^{\xi s} \bigr\rvert \times \Bigl\lvert \E  e^{\frac{\xi^2}{2n} + \xi s} - \E  e^{\xi s}\Bigr\rvert   }{ \E e^{\frac{\xi^2}{2n} + \xi s } \E e^{\xi s} } \\
                                          & \leq C n^{-1},
    \end{equ}
    where $C > 0$ depends only on $L$.
    Moreover,
    \begin{align*}
        \psi'_{\infty} (s) = \frac{\E \bigl( \xi^2 e^{\xi s}  \bigr) }{ \E \bigl( e^{\xi s} \bigr)  } - \left\{ \frac{ \E \bigl( \xi e^{\xi s} \bigr) }{\E \bigl( e^{\xi s} \bigr)} \right\}^2 .
    \end{align*}
    Recalling that $|\xi | \leq L\,, |X_i| \leq L  $ and $|m - m_i | \leq L/n$,  and using the fact that
    \begin{align*}
        \sup_{|s|\leq L } \bigl\lvert \psi'_{\infty}(s) \bigr\rvert \leq L^2,
    \end{align*}
    we have
    \begin{equ}
        \label{l5.1-04}
        \bigl\lvert \psi_{\infty} (m) - \psi_{\infty}(m_i) \bigr\rvert \leq L^3 n^{-1}.
    \end{equ}
    Following \cref{l5.1-03,l5.1-04}, the inequality \cref{l5.1-01} holds.

    A similar argument implies that \cref{l5.1-02} holds as well.
\end{proof}

    Set
	\begin{align}
		\label{eq:sigma-F-5}
		\F = \sigma \left\{ X_1, \ldots, X_n \right\}.
	\end{align}	
	For any $1 \leq i, j \leq n$, define
    \begin{equ}
		\label{eq:FiFij}
        \F^{(i)} = \sigma \bigl(  \left\{ X_k , k \neq i \right\} \bigr) \quad \F^{(i,j)} = \sigma \bigl( \left\{ X_k, k \neq i, j \right\} \bigr).
    \end{equ}

\begin{lemma}
    Let $W = n^{-1 + \frac{1}{2k}} \sum_{i = 1}^{n} X_i$, $G(w) = h^{(2k)} (0) w^{2k} / (2k)!$,
    and
    \begin{align*}
        \zeta (w, s) =
        \begin{cases}
            e^{G(w) - G(w- s)}, & w > s, \\
            e^{G(w)} ,          & 0 \leq w \leq s, \\
            1,                  & w < 0.
        \end{cases}
    \end{align*}
    Suppose \cref{d0,d1,r2.2-a,r2.2-b} are satisfied.  Then,
    we have
            \begin{equ}
                \label{l5.2-a}
                \E \biggl\lvert \frac{1}{n} \sum_{i = 1}^n \Bigl(X_i^2 - \conepb{X_i^2}{\F^{(i)}} \Bigr) \biggr\rvert \zeta (W, s)  \leq Cn^{-1/k} ( 1 + |s|^2 ) \, \E \, \zeta (W, s) ,
            \end{equ}
            where $C$ is a positive constant depending only on $\rho$.

    \label{lem:5.2}
\end{lemma}

We are now ready to prove \cref{thm:3.2}.

\begin{proof}
    [Proof of \cref{thm:3.2}]
    \def\F{\mathcal{F}}
    We first construct the exchangeable pair of $W$. For each $ 1 \leq i \leq n$, let $X_i'$ follow the conditional distribution of $X_i$ given $\left\{ X_j, j \neq i \right\}$, and be conditionally independent of $X_i$ given $\left\{ X_j, j \neq i \right\}$. Let $I$ be a random index uniformly distributed among $\left\{ 1, 2, \dots, n \right\}$, independent of all other random variables. Define $S_n' = S_n - X_I + X_I'$ and $W' = n^{-\frac{1}{2k}} S_n'$. Then $(W, W') $ is an exchangeable pair.
	Let \(\mathcal{F}, \F^{(i)}\) and \(\F^{(i,j)}\) be defined as in \cref{eq:sigma-F-5,eq:FiFij}.
	Let \(\psi_n, \psi_{\infty}, \phi_n\) and \(\phi_{\infty }\) be as defined in \cref{lem:5.1}.
    We have
    \begin{align}
		\conepb{X_i'}{ \F^{(i)} } = \conepb{X_i}{\F^{(i)}} = \psi_n (m_i(\mathbf{X})),
        \label{eq:gcw:t-01}
    \end{align}
    where
	$m_i (\mathbf{X})= \frac{1}{n} \sum_{j \neq i } X_j$. 

    Thus,
    \begin{equ}
        \conep{ X_I - X_I' }{\F} & = \frac{1}{n} \sum_{i = 1}^n \conep{X_i - X_i' }{ \F } \\
                                 & = m(\mathbf{X}) - \frac{1}{n} \sum_{i = 1}^n \conepb{X_i'}{ \F^{(i)} }  \\
								 & = m(\mathbf{X}) - \frac{1}{n} \sum_{i = 1}^n \psi_n ( m_i(\mathbf{X}) ) \\
                                 & = m(\mathbf{X}) - \psi_{\infty} ( m(\mathbf{X}) )  + r(\mathbf{X}) \\
                                 & = h'(m(\mathbf{X})) + r(\mathbf{X}),
                                 \label{eq:gcw:t-02}
    \end{equ}
	where \(m(\mathbf{X}) = (1/n) \sum_{i = 1}^n X_i\), $h$ is as defined in \cref{eq:gcw03}, and
    \begin{equ}
        \label{eq:gcw:t-03}
		r(\mathbf{X}) = \frac{1}{n} \sum_{i = 1}^n \left\{ \psi_{\infty} ( m(\mathbf{X}) ) - \psi_n ( m_i (\mathbf{X})) \right\}.
    \end{equ}
    By \cref{lem:5.1}, we have
    \begin{align*}
        |r(\mathbf{X})| \leq C n^{-1},
    \end{align*}
    where $C > 0$ is a constant depending only on $\rho$.
    As $\rho$ is symmetric, $h^{(2k + 1)}(0) = 0$.  By the Taylor expansion, for $|w| \leq L,$
    \begin{align*}
        |h'(w) - g(w)| \leq C |w|^{2k + 1},
    \end{align*}
    where $C >0$ is a constant depending only on $L$.
    Therefore,
    \begin{align*}
		\conep{W - W'}{\F} & = n^{-1 + \frac{1}{2k}} \conep{X_I - X_I'}{\F} \\
                                   & = n^{-1 + \frac{1}{2k}} \bigl( h'(m(\mathbf{X})) + r(m(\mathbf{X})) \bigr) \\
                                   & = \lambda ( g(W) + R(W) ),
    \end{align*}
    where $\lambda = n^{-2 + 1/k}$,
    \begin{align*}
        g(w) = \frac{ h^{(2k)}(0) }{ (2k - 1)! } w^{2k - 1}, \quad |R(w)| \leq C_1 n^{-1/k} (|w|^{2k + 1} + 1),
    \end{align*}
    where $C_1 > 0$ depends only on $\rho$.

    We now check the conditions \cref{r2.2-a,r2.2-b}.
    As $g(w) = \frac{h^{(2k)}(0)}{(k - 1)!} w^{2k - 1} $, then
    \begin{align*}
        |R(W)| \leq C_1 \Bigl( \frac{(k - 1)!}{ h^{(2k)}(0) } + 1 \Bigr) n^{-1/k} ( |W^2g(W)| + 1 ).
    \end{align*}
    Moreover, recalling that $|W| \leq L n^{\frac{1}{2k}}$, we have
    \begin{align*}
        |R(W)| \leq C_1 \left( n^{(2k-1)/2k } L^{2k+1} + 1 \right).
    \end{align*}
    Set
    \begin{align*}
        \kappa = \bigl( 2 C_1 \bigl( 1 + \frac{(k - 1)!}{h^{(2k)}(0)} \bigr) \bigr)^{-1/2} n^{\frac{1}{2k}},
    \end{align*}
    where $d_2 =  C_1 \left( n^{(2k-1)/2k } L^{2k+1} + \frac{(k - 1)!}{ h^{(2k)}(0) } + 2 \right) $.
    Thus,
    \begin{equ}
		\label{eq:5RW}
        |R(W)| \leq \frac{1}{2} ( |g(W)| + 1) +  d_2 I(|W| \geq \kappa).
    \end{equ}
    By \cite[][Propostion 6]{Ch10A}, for any $n \geq 1$ and $t \geq 0$,
    \begin{align*}
        \P (|W| \geq t) \leq 2 e^{-c_{\rho} t^{2k}},
    \end{align*}
    where $c_{\rho} > 0$ is a constant depending only on $\rho$.
    Note that $\delta = L n^{- 1 + \frac{1}{2k}}$ and by the definition of $g(\cdot)$, we have
    \begin{align*}
        s_0 = \max \left\{ s: \delta sg^2(s) \leq 1 \right\} = C_2 n^{(2k-1)/(2k (4k - 1))},
    \end{align*}
    where $C_2 > 0$ is a constant depending on $\rho$. Moreover, there exists a constant $d_1 > 0$ depending on $\rho$ such that $\delta|g(W)| \leq d_1$. Then, there exist positive  constants $C_3$ and $C_4$ depending on $\rho$ such that
    \begin{equ}
		\label{eq:6RW}
        d_2 e^{2 s_0 d_1^{-1} \delta^{-1}} \P (|W| \geq \kappa) \leq C_3  (n + 1) \exp\bigl\{ C_4 n^{2(k - 1)/(4k - 1)} - c_{\rho} n \bigr\} \leq d_3,
    \end{equ}
    where $d_3 > 0$ is a constant depending on $\rho$.  Thus the conditions \cref{d1,r2.2-a,r2.2-b}  hold.


    For the conditional second moment, by \cref{lem:5.1}, we have
    \begin{equ}
        \ML \conep{(X_I - X_I')^2}{ \F } \\
                                          & = \frac{1}{n} \sum_{i = 1}^n \conep{(X_i - X_i')^2}{ \F } \\
										  & = \frac{1}{n} \sum_{i = 1}^n X_i^2 - \frac{2}{n} \sum_{i = 1}^n X_i \psi_n (m_i(\mathbf{X})) + \frac{1}{n} \sum_{i = 1}^n \phi_n (m_i (\mathbf{X})) \\
										  & = \frac{1}{n} \sum_{i = 1}^n ( X_i^2 - \phi_n(m_i (\mathbf{X})) ) - 2 m(\mathbf{X}) \psi_{\infty} (m (\mathbf{X})) \\
                                          & \quad \quad + 2\phi_{\infty} ( m(\mathbf{X}) ) + r_2 (\mathbf{X}),
                                     \label{eq:gcw:t-04}
    \end{equ}
    where \(\psi_n, \phi_n\) and \(\phi_{\infty}\) are as defined  in \cref{lem:5.1}.
    By the Taylor expansion, we have
    \begin{equ}
        \label{eq:gcw:t-05}
        \bigl\lvert \phi_{\infty} ( m(\mathbf{X}) ) - 1  \bigr\rvert  = \bigl\lvert h''( m(\mathbf{X}) ) \bigr\rvert \leq C n^{-1 + 1/k} ( 1 + |W|^{2k-2} ),
    \end{equ}
    and
    \begin{equ}
        \label{eq:gcw:t-06}
        \bigl\lvert m(\mathbf{X}) \psi_{\infty} (m(\mathbf{X})) \bigr\rvert \leq C n^{-1/k} |W|^2 + C n^{-1} |W|^{2k},
    \end{equ}
    where $C>0$ is a constant depending only on $\rho$. By the definition of $(W, W')$ and \cref{eq:gcw:t-04,eq:gcw:t-05,eq:gcw:t-06}, with $\lambda = n^{-2+1/k}$, we have
    \begin{align*}
        \ML \biggl\lvert \frac{1}{2 \lambda} \conepb{(W - W')^2}{\F} - 1 \biggr\rvert\\
        & = \Bigl\lvert \frac{1}{2} \conepb{(X_I - X_I')^2}{\F} - 1 \Bigr\rvert \\
		& \leq \frac{1}{2} \biggl\lvert \frac{1}{n} \sum_{i = 1}^n ( X_i^2 - \phi_n (m_i(\mathbf{X})) ) \biggr\rvert + C n^{-1/k} \bigl(  1 + |W|^2 \bigr).
    \end{align*}
    Moreover,
    as $|X_i| \leq L$, we have
    \begin{align*}
        \biggl\lvert \frac{1}{2 \lambda} \conepb{(W - W')^2}{\F} - 1 \biggr\rvert \leq 2L^2 + 1 :=  d_0.
    \end{align*}
    Then \cref{d0} holds.
    By \cref{lem:5.2}, we have the condition \cref{r2.1-c} in \cref{r2.1} is satisfied.

    Hence, we have \cref{de2,d0,d1} and the conditions in \cref{r2.1,r2.2} are satisfied with
    $\delta_1 = \delta_2 = C n^{-1/k}$, $\tau_1 = \frac{2}{2k-1}$, and $ \tau_2 = 1 + \frac{2}{2k-1}$. By \cref{r2.1,r2.2}, we complete the proof of \cref{thm:3.2}.
\end{proof}
It suffices to proof \cref{lem:5.2}.
\begin{proof}
    [Proof of \cref{lem:5.2}]
    In this proof, we denote $C$ by a general positive constant depending only on $\rho$.
    By the Cauchy inequality, we have
    \begin{equ}
        \label{l5.2-01}
        \ML \E \biggl\lvert \frac{1}{n} \sum_{i = 1}^n \Bigl(X_i^2 - \conepb{X_i^2}{\F^{(i)}} \Bigr) \biggr\rvert \zeta(W, s) \\
        & \leq \biggl( \E \biggl\lvert \frac{1}{n} \sum_{i = 1}^n \Bigl(X_i^2 - \conepb{X_i^2}{\F^{(i)}} \Bigr) \biggr\rvert^2  \zeta(W, s)  \times \E\, \zeta(W, s)\biggr)^{1/2}.
    \end{equ}
    Expand the square term, and we have
    \begin{equ}
        \label{l5.2-02}
        \E \biggl\lvert \frac{1}{n} \sum_{i = 1}^n \Bigl(X_i^2 - \conepb{X_i^2}{\F^{(i)}} \Bigr) \biggr\rvert^2  \zeta(W, s)
        & = H_1 + H_2,
    \end{equ}
    where
    \begin{align*}
        H_1 & =   \frac{1}{n^2} \sum_{i = 1}^n \E \bigl\{\bigl(X_i^2 - \conepb{X_i^2}{\F^{(i)}} \bigr)^2 \zeta(W, s) \bigr\}, \\
        H_2 & =   \frac{1}{n^2} \sum_{i \neq j} \E \bigl\{\bigl(X_i^2 - \conepb{X_i^2}{\F^{(i)}} \bigr)  \bigl(X_j^2 - \conepb{X_j^2}{\F^{(j)}} \bigr) \zeta(W, s) \bigr\} .
    \end{align*}

    Recalling that $|X_i| \leq L$, we have
    \begin{equ}
        \label{l5.2-03}
        H_1 \leq 4 L^4 n^{-1} \E \zeta(W, s).
    \end{equ}

    As for $H_2,$ we first introduce some notations.
    For $i \neq j$, let $\E^{(i,j)}$ denote the conditional expectation given $\F^{(i,j)}$, where \(\F^{(i,j)}\) is as in \cref{eq:FiFij}.
    Note that
    \begin{align*}
        \E^{(i,j)} (X_i^2) & = \frac{ \iint x^2 \exp \Bigl( \frac{1}{2n} (x + y)^2 + (x + y) m_{ij}  \Bigr) \dd \rho(x) \dd \rho (y) }{  \iint \exp \Bigl( \frac{1}{2n} (x + y)^2 + (x + y) m_{ij}  \Bigr) \dd \rho(x) \dd \rho (y)  },
    \end{align*}
	where $m_{ij}:= m_{ij}(\mathbf{X}) = \frac{1}{n}\sum_{k\neq i, j} X_k$. Similar to \cref{lem:5.1}, we have for any $i \neq j$,
    \begin{equ}
        \label{l5.2-04}
        \Bigl\lvert \conepb{X_i^2}{ \F^{(i)} } - \E^{(i,j)} (X_i^2)   \Bigr\rvert \leq C n^{-1},
    \end{equ}
    where $C > 0$ depends only on $L$. Define
    \begin{equ}
        \label{l5.2-05}
    H_3 & =   \frac{1}{n^2} \sum_{i \neq j} \E \bigl\{\bigl(X_i^2 - \E^{(i,j)} (X_i^2) \bigr)  \bigl(X_j^2 - \E^{(i,j)} (X_j^2)  \bigr) \zeta(W, s) \bigr\} ,
    \end{equ}
    and then by \cref{l5.2-04,l5.2-05}, we have
    \begin{equ}
        \label{l5.2-06}
        \bigl\lvert H_2 - H_3 \bigr\rvert \leq C n^{-1} \E \zeta(W, s).
    \end{equ}

    We now move to give the bound of $H_3$. Define
    \begin{align*}
        W^{(i,j)} = W - n^{-1 + \frac{1}{2k}} (X_i + X_j).
    \end{align*}
    Let
    \begin{align*}
        q(w, s) =
        \begin{cases}
            G(w) - G(w - s), & w > s, \\
            G(w) , & 0 \leq w \leq s, \\
            0, & w < 0,
        \end{cases}
    \end{align*}
    and then $q(w , s) = \log \zeta (w, s)$ and $q'(w)$ is continuous on $\R$.
    Therefore, by the Taylor expansion, we have
    \begin{equ}
        \label{l5.2-07}
        q(W) - q(W^{(i,j)})
        & = (W - W^{(i,j)})  q' (W^{(i,j)}) \\
        & \quad + \frac{1}{2} (W - W^{(i,j)})^2 q'' (w_0)  ,
    \end{equ}
    where $w_0$ belongs to either  $(W, W^{(i,j)})$ or $(W^{(i,j)}, W)$. Note that $G(w) = C w^{2k}$ for some constant $C$, $|W| \leq L n^{\frac{1}{2k}}$ and $|W - W^{(i,j)}| \leq 2 L n^{-1 + \frac{1}{2k}}$. By the definition of $q$, we have
    \begin{align}
        \ML \Bigl\lvert  \bigl(W - W^{(i,j)}\bigr) q'\bigl(W^{(i,j)}\bigr) \Bigr\rvert \nonumber \\
        \label{l5.2-08}
        & \leq C n^{-1+\frac{1}{2k}} \bigl\lvert W^{(i,j)} \bigr\rvert^{2k-1}  \\
        & \leq C n^{-1+\frac{1}{2k}} \bigl( \bigl\lvert W \bigr\rvert^{2k - 1} + 1 \bigr) \nonumber
        \intertext{and}
        \ML \Bigl\lvert  \frac{1}{2} \bigl(W - W^{(i,j)}\bigr)^2 q''(w_0)  \Bigr\rvert \leq C n^{-1},
        \label{l5.2-09}
    \end{align}
    where $C$ depends only on $\rho$. Therefore, by \crefrange{l5.2-07}{l5.2-09} and using the fact that $|W| \leq Ln^{\frac{1}{2k}}$,
    we have
    \begin{equ}
        \label{l5.2-10}
        \bigl\lvert q(W) - q\bigl(W^{(i,j)}\bigr)  \bigr\rvert & \leq C n^{-1 + \frac{1}{2k}} ( |W|^{2k - 1} + 1 ) \\
        & \leq C.
    \end{equ}
    Observe that
    \begin{gather}
        \E^{(i,j)} \bigl\{\bigl( X_i^2 - \E^{(i,j)} (X_i^2) \bigr) \bigl( X_j^2 - \E^{(i,j)} (X_j^2)  \bigr) \zeta(W, s) \bigr\}
        = \zeta\bigl( W^{(i,j)}\bigr) M^{(i,j)}, \label{l5.2-11}
        \intertext{where}
        M^{(i,j) }  = \E^{(i,j)} \Bigl\{\bigl( X_i^2 - \E^{(i,j)} (X_i^2) \bigr) \bigl( X_j^2 - \E^{(i,j)} (X_j^2)  \bigr) e^{q(W) - q(W^{(i,j)})} \Bigr\}.  \nonumber
    \end{gather}
    Applying the Taylor expansion to the exponential function, we have
    \begin{equ}
        \label{l5.2-12}
        M^{(i,j)} = M_1^{(i,j)} + M_{2}^{(i,j)} + M_{3}^{(i,j)},
    \end{equ}
    where
    \begin{align*}
        M_1^{(i,j)} & =  \E^{(i,j)}\{ (X_i^2 - \E^{(i,j)} X_i^2) (X_j^2 - \E^{(i,j)} X_j^2)\}, \\
        M_2^{(i,j)} & =   \E^{(i,j)} \Bigl( (X_i^2 - \E^{(i,j)} X_i^2) (X_j^2 - \E^{(i,j)} X_j^2) \bigl\{ q (W) - q\bigl(W^{(i,j)}\bigr) \bigr\} \Bigr),
        \intertext{and}
        M_3^{(i,j)}  & = M^{(i,j) } - M_1^{(i,j)} - M_2^{(i,j)}.
    \end{align*}

    For $M_1^{(i,j)}$, since $ \E^{(i,j)} X_i^2 = \E^{(i,j)} X_j^2 $, we have
    \begin{align*}
        M_1^{(i,j)} & =  \E^{(i,j)} X_i^2 X_j^2 - \E^{(i,j)} X_i^2 \E^{(i,j)} X_j^2 \\
                    & = \frac{\iint x^2y^2 \exp \bigl( \frac{1}{2n}(x + y)^2 + (x + y) m_{ij} \bigr) \dd \rho(x) \dd \rho(y)}{ \iint \exp \bigl( \frac{1}{2n}(x + y)^2 + (x + y) m_{ij} \bigr) \dd \rho(x) \dd \rho(y) } \\
                    & \quad \quad -  \biggl( \frac{\iint x^2 \exp \bigl( \frac{1}{2n}(x + y)^2 + (x + y) m_{ij} \bigr) \dd \rho(x) \dd \rho(y)}{ \iint \exp \bigl( \frac{1}{2n}(x + y)^2 + (x + y) m_{ij} \bigr) \dd \rho(x) \dd \rho(y) } \biggr)^2  \\
                    & = M_{11}^{(i,j)} + M_{12}^{(i,j)},
    \end{align*}
    where
    \begin{equ}
        \label{l5.2-13}
        M_{11}^{(i,j)} & = \frac{\iint x^2y^2 \exp \bigl((x + y) m_{ij} \bigr) \dd \rho(x) \dd \rho(y)}{ \iint \exp \bigl((x + y) m_{ij} \bigr) \dd \rho(x) \dd \rho(y) }    \\
                       & \quad\quad -  \biggl( \frac{\iint x^2 \exp \bigl((x + y) m_{ij} \bigr) \dd \rho(x) \dd \rho(y)}{ \iint \exp \bigl((x + y) m_{ij} \bigr) \dd \rho(x) \dd \rho(y) } \biggr)^2\\
                       & = 0,
    \end{equ}
    and $M_{12}^{(i,j)} = M_1^{(i,j)} - M_{11}^{(i,j)}$. Similar to \cref{lem:5.1}, we have
    \begin{equ}
        \label{l5.2-14}
        |M_{12}^{(i,j)}| \leq C n^{-1}.
    \end{equ}
    By \cref{l5.2-13,l5.2-14}, we have
    \begin{equ}
        \label{l5.2-15}
        \bigl\lvert M_1^{(i,j)} \bigr\rvert \leq C n^{-1}.
    \end{equ}

    For $M_2^{(i,j)}$, by \cref{l5.2-07,l5.2-09}, we have
    \begin{align*}
        M_2^{(i,j)} & =  M_{21}^{(i,j) } + M_{22}^{(i,j)} ,
    \end{align*}
    where
    \begin{align*}
        M_{21}^{(i,j) } & =   n^{-1 + \frac{1}{2k}} q' \bigl( W^{(i,j)} \bigr) \E^{(i,j)} \{ (X_i^2 - \E^{(i,j)} X_i^2 ) (X_j^2 - \E^{(i,j)} X_j^2) (X_i + X_j)\} , \\
        M_{22}^{(i,j)} & =  \frac{1}{2} \E^{(i,j)} \{ (X_i^2 - \E^{(i,j)} X_i^2 ) (X_j^2 - \E^{(i,j)} X_j^2) (W - W^{(i,j)})^2 q''(w_0) \},
    \end{align*}
    and $w_0$ is as defined in \cref{l5.2-07}. By \cref{l5.2-09}, and recalling that $|X_i| \leq L$, we have
    \begin{align*}
        \bigl\lvert M_{22}^{(i,j)} \bigr\rvert \leq C n^{-1}.
    \end{align*}
    Similar to \cref{l5.2-15}, we have
    \begin{align*}
        \bigl\lvert \E^{(i,j)} \{(X_i^2 - \E^{(i,j)} X_i^2 ) (X_j^2 - \E^{(i,j)} X_j^2) (X_i + X_j) \}
        \bigr\rvert \leq C n^{-1}.
    \end{align*}
    Moreover, recalling  that $|W^{(i,j)}| \leq L n^{\frac{1}{2k}}$ and $|q'(W^{(i,j)})| \leq C n^{1 - \frac{1}{2k}}$, we have
    \begin{align*}
        \bigl\lvert M_{21}^{(i,j)} \bigr\rvert \leq C n^{-1}.
    \end{align*}
    Thus,
    \begin{equ}
        \label{l5.2-16}
        |M_2^{(i,j)}| \leq C n^{-1}.
    \end{equ}
    For $M_{3}^{(i,j)}$, by the Taylor expansion, noting again that $k \geq 2,~|W| \leq L n^{\frac{1}{2k}} $ and $|X_i| \leq L$ for $1 \leq i\leq n$, and by \cref{l5.2-08,l5.2-09},  we have
    \begin{equ}
        \label{l5.2-17}
        \bigl\lvert M_3^{(i,j)} \bigr\rvert 
        & \leq C \bigl\lvert q(W) - q(W^{(i,j)}) \bigr\rvert^2 e^{ q(W) - q(W^{(i,j)})  }\\
        & \leq C n^{-2 + 1/k} \bigl( |W|^{4k - 2} + 1 \bigr) \\
        & \leq C n^{-2/k} \bigl( |W|^{4} + 1 \bigr).
    \end{equ}
    By \cref{l5.2-12,l5.2-15,l5.2-16,l5.2-17}, we have
    \begin{align*}
        |M^{(i,j)} | \leq C n^{-2/k} ( |W|^4 + 1 ),
    \end{align*}
    substituting which to \cref{l5.2-11}, we have
    \begin{equ}
        & \E \bigl\lvert \E^{(i,j)} \bigl\{ \bigl( X_i^2 - \E^{(i,j)} (X_i^2) \bigr) \bigl( X_j^2 - \E^{(i,j)} (X_j^2)  \bigr) \zeta(W, s) \bigr\} \bigr\rvert \\
        & \leq C n^{-2/k} \E \bigl\{ (|W|^4 + 1 )  \zeta\bigl( W^{(i,j)}\bigr) \bigr\} \\
        & \leq C n^{-2/k} \E \bigl\{( |W|^4 + 1 ) \zeta(W, s)\bigr\}\\
        & \leq C n^{-2/k} (1 + s^4) \E \zeta(W, s),
        \label{l5.2-18}
    \end{equ}
    where in the last inequality we used \cref{lem:4.10} recalling the fact that \cref{eq:5RW,eq:6RW} are satisfied. By \cref{l5.2-18}, we have the term $H_3$ in \cref{l5.2-05} can be bounded by
    \begin{equ}
        \label{l5.2-19}
        |H_3| \leq C n^{-2/k} (1 + s^4) \E \zeta(W, s) . 
    \end{equ}
    By \cref{l5.2-01,l5.2-02,l5.2-03,l5.2-06,l5.2-19}, we complete the proof of \cref{l5.2-a}.
    \qedhere
\end{proof}

\subsection{Proof of Theorem 3.2}
In this subsection, we use \cref{r2.2} to prove the result.

\begin{proof}[Proof of \cref{thm:3.3}]

	For any $\sigma \in \Sigma, uv\in D$ and $s,t \in \{ 0,1 \}$, let $\sigma_{uv}^{st}$ denote the configuration $\tau\in \Sigma$, such that $\tau_i = \sigma_i$ for $i\neq u,v$ and $\tau_u = s$, $\tau_v= t$. Let $(\sigma_u', \sigma_v')$ be independent of $(\sigma_u,\sigma_v)$ and follow the conditional distribution
  \[
  \P(\sigma_u'=s, \sigma_v' = t| \sigma) = \frac{p(\sigma_{uv}^{st}) }{ \sum_{s,t\in \{0,1\}} p(\sigma_{uv}^{st}) }.
  \]
  Let $M = \sum_{i=1}^n \sigma_i$ and $M' = M - \sigma_u - \sigma_v + \sigma_u' + \sigma_v'$. Then, by \citet*{Ch16L}, $(M,M')$ is exchangeable. Also, by \citet*[][Proposition 2]{Ch16L}, we have
  \begin{align}
  \E\bigl(M-M'|\sigma \bigr) = {}&    L_1(m(\sigma))+ R_1(m(\sigma)),\lbl{eq:m1}\\
  \E\bigl((M-M')^2 | \sigma \bigr) = {}&   L_2(m(\sigma)) + R_2( m (\sigma)), \lbl{eq:m2}
  \end{align}
  where $m(\sigma) = M/n$ and
  \begin{eqnarray*}
      &  L_1(x) =  \frac{2(1-x)(x^2 - (1-x) e^{2\tau(x)}) }{ (1-x) + e^{2\tau(x)}}, \text{ for }  0 < x < 1, \\
      &  L_2(x) =  \frac{4(1-x)(x^2 + (1-x) e^{2\tau(x)}) }{ (1-x) + e^{2\tau(x)}}, \text{ for }  0 < x < 1, \\
    &  |R_1(x)| + |R_2(x)| \leq \frac{C}{ n}
  \end{eqnarray*}
  for some constant $C$.
  Next, we consider two cases. In the first case, $(J,h)\not\in\Gamma \cup \{(J_c,h_c)\}$, and in the second case, $(J,h)=(J_c,h_c).$
  \begin{description}[leftmargin=0cm]
      \item[\it Case 1. ]{\it  When $(J,h)\not\in \Gamma \cup\{(J_c,h_c)\}$. } Define $W = n^{-1/2}(M - n m_0)$ and $W'=n^{-1/2}(M' - n m_0)$; then, $(W,W')$ is also an exchangeable pair. Moreover, $$|W - W'| \leq 2 n^{-1/2}=:\delta. $$

    Note that $L_1(m_0)=0$ by observing $m_0^2 = (1-m_0)e^{2\tau(m_0)}$. Moreover, we have
    \[
    L_1'(m_0) = \frac{1}{ 2\lambda_0 } L_2(m_0) >0,
    \]
    where $\lambda_0  = (-1/H''(m_0)) - (1/2J)>0$.
    By the Taylor expansion, we have
    \[
    L_1(m(\sigma)) = L_1'(m_0) (m(\sigma)-m_0) +  \int_{m_0}^{m(\sigma)} L_1''(s)(m(\sigma ) - s ) ds.
    \]
    Let $\lambda = L_2(m_0)/(2n)$, and we have
    \[
    n^{-1/2} L_1(m(\sigma)) = \lambda \left(\lambda_0^{-1} W + r(W) \right),
    \]
    where
    \[
    r(W) = 2n^{1/2}L^{-1}_2(m_0) \int_{m_0}^{m(\sigma)} L_1''(s)(m(\sigma)-s) ds.
    \]
    Therefore, together with the definition of $(W,W')$ and \eqref{eq:m1}, we have
    \[
        \E(W-W'|W)= n^{-1/2}(L_1(m(\sigma)) + R_1(m(\sigma))) =\lambda (g(W) + R(W)),
    \]
    where
    \[
        g(W) = W/\lambda_0 \quad\mbox{ and }  \quad R(W) = r(W) + \frac{2 n^{1/2}}{L_2(m_0)}R_1(m(\sigma)).
    \]
    Thus, \crefrange{icon1}{icon4} hold for $g(w) = w / \lambda_0$. Furthermore, $\delta|g(W)| \leq {2} / {\lambda_0}$, as $n^{-1/2} |W| \leq 1$. 

    By \citet[][Lemma 1]{Ch16L}, there exist constants $C_0, C_1 > 0$ such that
    \begin{align}
        \label{eq:pr-03}
    |R(W)|  \leq {} &  C_0n^{-1/2}(W^2 + 1), \nn
    \end{align}
    and
    \[
    \left| \frac{1}{ 2\lambda } \E((W-W')^2 | W) - 1  \right| \leq C_1 n^{-1/2}(|W| + 1)
    \]
    and $|\hat{K}_1| = \frac{\Delta^2}{2\lambda} \leq 4 / L_2(m_0)$.
    Therefore, \cref{de1,de2,d0,d1} are satisfied with $\tau_1 = 1, \tau_2 = 2, \delta_1 = \delta_2 = O(1) n^{-1/2}$ and $d_0 = 4 / L_2(m_0)$ and $d_1 = 2 / \lambda_0$.

    It suffices to prove \cref{r2.2-a,r2.2-b}. By \cref{eq:pr-03}, we have
    for $|W| \leq \frac{\sqrt{n}}{2\lambda_0C_0}$,
    \begin{align}
        |R(W)| \leq \frac{1}{2} (|g(W)| + 1),
        \label{eq:pr-01}
    \end{align}
    and for $|W| > \frac{ \sqrt{n}}{2 \lambda_0 C_0}$, recalling that $|W| \leq 1$, we have $|R(W)| \leq C_0 (\sqrt{n} + 1)$. Then, \cref{r2.2-a} holds with $\alpha = 1/2, \, d_2 = C_0 (\sqrt{n} + 1)$ and $\kappa = \sqrt{n} / ( 2 \lambda_0 C_0)$.
    By \citet[][Lemma 2]{Ch16L}, when $(J, h) \not\in \Gamma \cup \left\{ (J_c, h_c) \right\},$ for any $u > 0$, there exists a constant $\eta > 0$ such that
    \begin{align*}
        \P(|m(\sigma)-m_0| \geq u) \leq C e^{-n \eta}
    \end{align*}
    for some constant $C$.
    Hence,
    \begin{align*}
        d_2 \P (|W| > \kappa ) & \leq C ( \sqrt{n} + 1) e^{-n \eta} .
    \end{align*}
    Note that $s_0 = \max \left\{ s : \delta sg^2(s) \leq 1 \right\}$, $g(w) = w / \lambda_0$, $d_1 = \frac{2}{\lambda_0}$  and $\delta = 2 n^{-1/2}$, then $s_0 = (\lambda_0/2)^{1/3} n^{1/6}$. Therefore, \cref{r2.2-b} is satisfied.
    By \cref{r2.2}, we have
    \begin{align*}
        \frac{\P(W > z)}{\P(Z_0 > z)} = 1 + O(1) n^{-1/2} (1 + z^3)
    \end{align*}
    for $0 \leq z \leq n^{1/6}$.


    \item[\it Case 2. ]{\it  When $(J,h) = (J_c,h_c)$. } Define $W=n^{-3/4}(M-n m_c)$ and $W'=n^{-3/4}(M'-nm_c)$; then, $(W,W')$ is an exchangeable pair. By \eqref{eq:m1}, we have
        \[
        \E(W-W'|W) = n^{-3/4} (L_1(m(\sigma)) + R_1(m(\sigma))).
        \]
        By \citet*[][p. 14]{Ch16L}, 
        we have
        \[
			L_1(m_c)= L_1'(m_c) = L_1''(m_c) =0, \quad L_1^{(3)} (m_c)= \frac{\lambda_c}{2}L_2(m_c),
        \]
        where $\lambda_c$ is given in \eqref{Yden}. Then, by the Taylor expansion, we have
        \[
        L_1(m(\sigma)) = \frac{L_1^{(3)}(m_c) }{ 6} (m(\sigma)-m_c)^3 +
        \frac{1}{ 6}\int_{m_c}^{m(\sigma)} L_1^{(4)} (s) (m(\sigma)-s)^3 ds.
        \]
        Then, taking $\lambda= L_2(m_c)/ (2n^{3/2})$, by \citet*[][Lemma 1]{Ch16L}, 
        we have
        \begin{align*}
        \E(W-W'|W) = \lambda (g(W) + R(W)), \nn
        \end{align*}
        where $g(W) = (\lambda_c/6)W^3$ and
        \[
            R(W) = \frac{ n^{3/4}}{2 L_2(m_c) }\int_{m_0}^{m(\sigma)} L_1^{(4)} (s) (m(\sigma)-m_c)^3 ds + \frac{2 n^{3/4}}{L_2(m_c)}  R_1(W).
        \]
        Hence, $G(w) = \frac{\lambda_c}{24 } w^4$.
        Based again on  \citet*[][Lemma 1]{Ch16L}, 
        for some constant $C$, we have
        \begin{align}
        |R(W)| \leq C n^{-1/4} (|W|^4 + 1) \leq Cn^{-1/4}(|g(W)|^{4/3} + 1 ),
        \label{eq:pr-r2}
        \end{align}
        and
        \[
        \left| \frac{1}{ 2 \lambda} \E((W-W')^2 | W) - 1  \right| \leq Cn^{-1/4}(|g(W)|^{1/3} + 1).
        \]

        As $|W-W'|\leq 2n^{-3/4}$ and $|W| \leq Cn^{1/4},$ it follows that there exist constants \(d_0\) and \(d_1\) such that $n^{-3/4}|g(W)| \leq d_1$ and $\hatK_1 =  (W-W')^2/(2\lambda) \leq d_0$. Thus, \cref{d0,d1} are satisfied. Furthermore, \cref{de1,de2} hold with $\delta = 2n^{-3/4}, \delta_1 = \delta_2 = O(1) n^{-1/4}$ and $\tau_1 = 1/3, \tau_2 = 4/3$.
        It suffices to show that \cref{r2.2-a} and \cref{r2.2-b} are satisfied.
        By \cref{eq:pr-r2}, there exists a constant $c > 0$ such that for $|W| \leq c n^{1/4}$,
        \begin{align*}
            |R(W)| \leq \frac{1}{2} (|g(W)| + 1).
        \end{align*}
        For $|W| \geq cn^{1/4}$, noting that $|W| \leq C n^{1/4}$, we have $|R(W)| \leq C n^{3/4}$.
        Thus, \cref{r2.2-a} is satisfied with $\alpha = 1/2$, $d_2 = C n^{3/4}$ and $\kappa = c n^{1/4}$.
        Furthermore, as $\delta = 2 n^{-3/4}$ and $g(w) = (\lambda_c/6)w^3$, we have $s_0 = (18/\lambda_c)^{1/7} n^{3/28}$.
        In addition, by  \citet*[][Lemma 2]{Ch16L}, 
        when $(J, h) = (J_c, h_c)$, for any $u > 0$, there exists a constant $\eta > 0$ such that
        \begin{align*}
            \P(|m(\sigma)-m_c| \geq u) \leq C e^{-n \eta}.
        \end{align*}
         Thus,
         \begin{align*}
             d_2 \P(|W| \geq \kappa) & \leq C n^{3/4} e^{-n\eta} \leq C e^{ - 2 s_0 d_1^{-1}\delta^{-1}}.
         \end{align*}
         Then, \cref{r2.2-b} holds.
         By \cref{r2.2}, we complete the proof of \cref{thm:3.3}. \qedhere
\end{description}
\end{proof}

\section*{Acknowledgements}
We would like to thank the  referees for their helpful comments which led to a much improved presentation of the paper.

\end{document}